\documentclass[a4paper,11pt]{amsart}
\usepackage{amsmath,amsthm,amssymb,amsfonts,enumerate,color,esint}
\usepackage[pdftex]{graphicx}
\usepackage{float}
\usepackage[colorlinks=true, allcolors=blue]{hyperref}

\newtheorem{theorem}{Theorem}[section]
\newtheorem{lem}[theorem]{Lemma}
\newtheorem{lemma}[theorem]{Lemma}
\newtheorem{definition}[theorem]{Definition}

\newtheorem{corollary}[theorem]{Corollary}
\newtheorem{cor}[theorem]{Corollary}

\newtheorem{remark}[theorem]{Remark}
\newtheorem{rem}[theorem]{Remark}

\newtheorem*{theorem*}{Theorem}

\DeclareMathAlphabet{\mathpzc}{OT1}{pzc}{m}{it}

\renewcommand{\a}{\alpha}

\newcommand{\R}{\mathbb{R}}

\newcommand{\E}{\mathcal{E}}

\renewcommand{\a}{\alpha}

\newcommand{\Dl}{D_{\mathrm{left}}}
\newcommand{\Dr}{D_{\mathrm{right}}}

\def\supp{{\rm supp}}

\def\L{\mathcal{L}}
\def\S{\mathcal{S}}

\def\R{{\mathbb{R}}}

\numberwithin{equation}{section}

\allowdisplaybreaks

\author[M.~Mazzitelli]{Mart\'in Mazzitelli}
\address{Instituto Balseiro \\ CNEA-Universidad Nacional de Cuyo \\ Conicet \\ Argentina}
\email{martin.mazzitelli@ib.edu.ar}

\author[P.~R.~Stinga]{Pablo Ra\'ul Stinga}
\address{Department of Mathematics \\ Iowa State University \\ 396 Carver Hall \\
Ames, IA 50011 \\ United States of America}
\email{stinga@iastate.edu}

\author[J.~L.~Torrea]{Jos\'e L. Torrea}
\address{Departmento de Matem\'aticas \\ Universidad Aut\'onoma de Madrid \\ Campus de Cantoblanco \\
28049 Madrid \\ Spain}
\email{joseluis.torrea@uam.es}

\thanks{The first author was partially supported by CONICET PIP 11220200102366CO.
The second author was partially supported by Simons Foundation grant 580911.
The third author was partially supported by grant PGC2018-099124-B-100 (MINECO/FEDER) from Government of Spain.}

\begin{document}
%%%%%%%%%%%%%%%%%%%%%%%%%%%%%%%%%%%%%%%%%%%%%%%%%%%%%%

%%%%%%%%%%%%%%%%%%%%%%%%%%%%%%%%%%%%%%%%%%%%%%%%%%%%%%
\title[Fractional operators, inverse measures and polynomials]{Fractional powers of first order differential operators
and new families of polynomials associated to inverse measures}
%%%%%%%%%%%%%%%%%%%%%%%%%%%%%%%%%%%%%%%%%%%%%%%%%%%%%%

%%%%%%%%%%%%%%%%%%%%%%%%%%%%%%%%%%%%%%%%%%%%%%%%%%%%%%
\begin{abstract}
First, we establish the theory of fractional powers of first order differential operators
with zero order terms, obtaining PDE properties and analyzing the corresponding fractional Sobolev spaces.
In particular, our study shows that Lebesgue and Sobolev spaces with inverse measures
(like the inverse Gaussian measure) play a fundamental role in the theory of
fractional powers of the first order operators.
Second, and motivated in part by such a theory,
we lay out the foundations for the development of the harmonic analysis for \emph{inverse} measures.
We discover new families of polynomials related to the
inverse Gaussian, Laguerre, and Jacobi measures, and characterize them using
generating and Rodrigues formulas, and three-term recurrence relations. Moreover,
we prove boundedness of several fundamental singular integral operators in these inverse measure settings.
\end{abstract}
%%%%%%%%%%%%%%%%%%%%%%%%%%%%%%%%%%%%%%%%%%%%%%%%%%%%%%

\maketitle

%%%%%%%%%%%%%%%%%%%%%%%%%%%%%%%%%%%%%%%%%%%%%%%%%%%%%%
\section{Introduction}
%%%%%%%%%%%%%%%%%%%%%%%%%%%%%%%%%%%%%%%%%%%%%%%%%%%%%%

In recent years, singular integrals in the metric measure space $\R^d$ endowed with the
so-called \emph{inverse Gaussian measure} $d\gamma_{-d}=\pi^{n/2}e^{|x|^2}dx$ have been considered.
The underlying \emph{Laplacian} is $\mathbf{L}=\frac{1}{2}\Delta u+x\cdot\nabla u$,
which is essentially self-adjoint in $L^2(\R^d,\gamma_{-d})$.
For instance, maximal operators and multipliers were studied in \cite{Salogni}.
The analysis can also be seen as a toy model for a variety of non-doubling or non-Ahlfors regular
settings, where the theory of singular integral operators has not yet been fully established.
Other singular integrals and functional spaces for the inverse Gaussian measure
have been considered in \cite{jorge,Betancor-Rodriguez,italiano,Bruno-Sjogren},
just to mention a few works.

In the analysis of geometric flows, like the mean curvature flow, classes of special solutions
that can capture the behavior of the flow around singularities are considered. In particular, roughly speaking,
self-expanding solutions are those that retain the same shape and are expanding or growing.
It turns out that a self-expander $M$ is a manifold that minimizes the energy functional
$$E(M)=\int_Me^{|x|^2/2}\,dS$$
where $dS$ is the surface area element.~Spectral properties of the corresponding Euler--Lagrange equation, which is driven by $\mathbf{L}$,
and geometric consequences were derived in \cite{expanders}.

In fact, as we demonstrate in this paper, the inverse Gaussian measure space arises naturally as part of another
fundamental question, that is, in the analysis of fractional powers of first order differential operators. Furthermore,
we show that the inverse Gaussian  
measure space is just one particular example of a much broader theory of special functions and
new orthogonal families, which
we develop in detail here, and include inverse Laguerre and  Jacobi measures. 
Inverse Gegenbauer and ultraspherical measures and corresponding polynomials will be particular cases
of the Jacobi setting.

Consider the first order differential operator
$$Au=u'+a(x)u$$
acting on functions $u:\R\to\R$, where $a=a(x)$ is a continuous function.
In the last decades, particular cases of these operators have appeared naturally in
harmonic analysis associated to generalized Laplacians, as playing a parallel role to that of
the standard derivative for the classical Laplace operator.
In particular, these first order operators and their powers have been crucial to define Riesz transforms
and  Littlewood--Paley square functions, and in the analysis of a priori estimates for generalized
Laplace equations, see, for example,
\cite{AbuTor, InGuTo, HarTorVivi,Muckenhoupt,Stinga-Torrea-Hermite,Thangavelu,Thangavelu-CPDE}.
Next, fix a point $x_0\in\R$ and define the weight function
$$\mathcal{E}(x)=\exp\bigg(-\int_{x_0}^xa(y)\,dy\bigg).$$
As we will show in detail in Section \ref{part:fractional}, the natural $L^p$ space for the analysis of 
$A$ and its fractional powers $A^\alpha$ and $A^{-\alpha}$, $\alpha>0$, is
$$L^p(\E^p)=\big\{u:\R\to\R:\E u\in L^p(\R)\big\}.$$
Indeed, for instance, the semigroup generated by $A$ is bounded in $L^p(\mathcal{E}^p)$, $1<p<\infty$,
and the correct Sobolev space associated to $A$ is given by $W^{1,p}_a=\big\{u\in L^p(\E^p):Au\in L^p(\E^p)\big\}$.

One of the most basic particular cases is
$$Au=u'-xu,$$
where we have chosen $a(x)=-x$. This is the 
natural first order derivative operator for the PDE and harmonic analysis of the
harmonic oscillator and Hermite function expansions,
see, for example, \cite{Bongioanni-Torrea,Stinga-Torrea-Hermite,Thangavelu,Thangavelu-CPDE}. If we choose $x_0=0$ then
$$\E(x)=e^{x^2/2}$$
and so
$$L^2(\E^2)=\big\{u:\R\to\R:e^{x^2/2}u\in L^2(\R)\big\}=L^2(\gamma_{-1}).$$

Obviously, a different choice of $a(x)$ gives a new natural $L^p$ space in which corresponding
operators are expected to be bounded. In this way, by considering the first order derivative operators
related to other special functions and orthogonal polynomials systems, we discover new families of
special functions and polynomials related to what we call the \emph{inverse} Laguerre, Gegenbauer,
and Jacobi measures, which are
$$e^{x}x^{-\alpha}\,dx\qquad x\in(0,\infty),~\alpha >-1;$$
$$(1-x^2)^{-\alpha+1/2}\,dx\qquad x \in (-1,1),~\alpha >-1;$$
and
$$(1-x)^{-\alpha}(1+x)^{-\beta}\,dx\qquad x \in (-1,1),~\alpha,\beta >-1;$$
respectively.

This paper has two parts. In Section \ref{part:fractional}, we present a quite complete theory of
fractional powers of general first order differential operators with zero order terms on the real line.
We obtain a variety of PDE and analysis results, including maximum principles, extension problems,
characterization of associated Sobolev spaces
by limits of fractional power operators in the spirit of Bourgain--Brezis--Mironescu \cite{Bourgain-Brezis-Mironescu},
and the fundamental theorem of calculus for our fractional derivatives and integrals. 
Our results generalize and complement the recent exhaustive analysis of fractional powers of the first derivative on the line
that has been developed in \cite{BerMarStiTor,mary}. In Section \ref{part:orthogonal},
we show that our general first order operators contain as particular cases operators
that have appeared in relation with Ornstein--Uhlenbeck, Hermite, Laguerre and Jacobi expansions.
We construct new families of inverse measures
and polynomials which satisfy three-term recurrence relations and can be characterized through new Rodrigues
and generating formulas. These families are also eigenfunctions of second order differential operators, which are self-adjoint with respect to the inverse measures. Our polynomials are not orthogonal with respect to the inverse measures. However, we prove that they are
orthogonal polynomial sequences with respect to a moment functional.
As applications of our methods to harmonic analysis for inverse measures, we present boundedness properties
of maximal operators, Riesz transforms and Littlewood--Paley square functions.

Sections \ref{part:fractional} and \ref{part:orthogonal} are interrelated and can also be seen as having independent interest. 
Readers interested in the PDE and harmonic analysis for fractional powers of first order differential operators
are invited to focus on Section \ref{part:fractional},
while those interested in special functions, orthogonal polynomials and harmonic analysis
in the inverse measure settings can jump directly to Section \ref{part:orthogonal}.

We will not describe in this introduction all the results of this paper. The reader can easily recognize
which are the main results of Section \ref{part:fractional}. Therefore,
we will only highlight the results of Section \ref{part:orthogonal}. These are directly motivated by our description of first order operators
of Section \ref{part:fractional}. The statements are related with the classical Hermite, Laguerre, and Jacobi polynomials. It is well known that the classical families are described by a Rodrigues formula, satisfy three-term recurrence relations and are eigenfunctions
(in fact, eigenpolynomials) of a certain second order differential operator. Each of these differential operators are formally self-adjoint respectively with respect to the measures
$e^{-x^2}dx$, $x \in \mathbb{R}$,  (Hermite); $ x^{\alpha} e^{-x}$, $ \alpha>-1$,  $x\in (0,\infty)$, (Laguerrre);
and  $ (1-x)^{\alpha}(1+x)^{\beta}$, $\alpha,\beta >-1$, $x \in (-1,1),$ (Jacobi).
We construct new families of polynomials for each \emph{inverse} case. All of them satisfy a three-term recurrence formula,
are described by a Rodrigues formula, have a generating function, and are eigenfunctions of differential operators.

\begin{theorem}[Inverse Gaussian]\label{inverHer}
Consider the family of   polynomials
$$\widetilde{H} _n(x) =(-1)^ne^{-x^2}  \frac{d^n e^{x^2}}{d x^n},\quad x\in\R,~n\geq0.$$
The function
$$w(x,t)= e^{t^2-2xt}\qquad x,t\in\R$$
is a generating function of the family $\{\widetilde{H} _n\}_{n\geq0}$ and the recurrence relation
\begin{eqnarray*}\widetilde{H} _{n+1}(x)+2x \widetilde{H} _n(x)-2n\widetilde{H} _{n-1}(x) =0,\quad n\geq1, \quad  \widetilde{H} _{0}(x)=1, \quad \widetilde{H} _{-1}(x)=0,
\end{eqnarray*}
is satisfied. Let
$$ \widetilde{\mathcal O} = \frac{d^2}{dx^2} + 2x \frac{d}{dx}.$$
Then  $\widetilde{\mathcal O} $ is formally self-adjoint with respect to the inverse Gaussian measure $d \gamma_{-1}(x) = \pi^{1/2}e^{x^2}dx$ on $\R$.
Moreover,  $ \widetilde{\mathcal O}\widetilde{H} _n(x)= 2n \widetilde{H} _n(x)$.
\end{theorem}

\begin{theorem}[Inverse Laguerre]\label{inverLag}
Fix $\alpha>-1$. Consider the family of polynomials
$$\widetilde{L}_n^\alpha(x)=  \frac{e^{-x}x^{\alpha}}{n!} \frac{d^n}{dx^n}(e^{x} x^{n-\alpha}),\quad x\in(0,\infty),~n\geq0.$$
The function
$$w( x,t) = (1-t)^{\alpha-1} e^{xt/(1-t)}\qquad x\in(0,\infty),~t<1$$
is a generating function of the family $\{\widetilde{L}_n^\alpha\}_{n\geq0}$ and the recurrence relation
$$(n+1)\widetilde{L}_{n+1}+(\alpha-1-x-2n)\widetilde{L}_n +(n-\alpha)\widetilde{L}_{n-1} =0,\, n\geq1, \quad \widetilde{L}_{0}(x)=1, \quad  \widetilde{L}_{-1}(x)=0, $$
is satisfied. Let
$$ \widetilde{\mathcal{L}}_\alpha = x\frac{d^2}{dx^2} + (-\alpha +1+x) \frac{d}{dx}.$$
Then $\widetilde{\mathcal{L}}_\alpha$ is formally self-adjoint with respect to the inverse Laguerre measure  $x^{-\a}e^{x}dx$ on $(0,\infty)$. 
Moreover, $\widetilde{\mathcal L}_\alpha\widetilde{L}_n = n \widetilde{L}_n$.
\end{theorem}

\begin{theorem}[Inverse Jacobi]\label{inverJac}
Fix $\alpha,\beta>-1$. Consider the family of polynomials
$$ \widetilde{P}_n^{\alpha,\beta}(x)=  (-1)^{n+1} \frac{2^{-n}}{n!}(1-x)^{\alpha}(1+x)^{\beta}\frac{d^n}{dx^n}((1-x)^{-\alpha+n} (1+x)^{-\beta +n}),\quad x \in (-1,1),~n\geq0.$$ 
The function
$$\omega(x,t)= \frac{(1-t+(1-2xt+t^2)^{1/2})^{\alpha} (1+t+(1-2xt+t^2)^{1/2})^{\beta}}{2^{\alpha+\beta} (1-2xt+t^2)^{1/2}},
\quad x\in (-1,1),~t\in \mathbb{C}$$
is a generating function of the family $\{\widetilde{P}_n^{\alpha,\beta}\}_{n\geq0}$ and the recurrence relation
\begin{eqnarray*}
\lefteqn{2(n+1)(n-\alpha-\beta+1)(2n-\alpha-\beta)\widetilde{P}_{n+1}^{\alpha,\beta}}\\ && \quad -(2n-\alpha-\beta+1)[(2n-\alpha-\beta+2)(2n-\alpha-\beta)x + \alpha^2-\beta^2]\widetilde{P}_n^{\alpha,\beta}\\
&&\quad  + 2(n-\alpha)(n-\beta)(2n-\alpha-\beta+2)\widetilde{P}_{n-1}^{\alpha,\beta}=0, \qquad n\geq 1,\\
&&  \widetilde{P}_{0}^{\alpha,\beta}(x)=1,\quad  \widetilde{P}_{-1}^{\alpha,\beta}(x)=0,
\end{eqnarray*}
is satisfied.
Let
$$\widetilde{\mathcal{G}}_{\alpha,\beta}  = (1-x^2)\frac{d^2}{dx^2} + 
\big( (\alpha-\beta) +(\alpha +\beta -2)x \big) \frac{d}{dx}.$$
Then $\widetilde{\mathcal{G}}_{\alpha,\beta}$ is formally self-adjoint with respect to the inverse Jacobi measure
$d\mu_{-\alpha,-\beta}(x) = (1-x)^{-\alpha}(1+x)^{-\beta}dx$  on $(-1,1)$.
Moreover, $\widetilde{\mathcal G}_{\alpha,\beta} \widetilde{P}_n^{\alpha,\beta}= -n(n-\alpha-\beta +1) \widetilde{P}_n^{\alpha,\beta}$.
\end{theorem}

In the previous statements, none of the new polynomials are orthogonal with respect to their natural inverse measure. However, by using a transference technique through Lebesgue measure that was introduced in \cite{AbuTor}, see also \cite{Abu-Macias}, we can find a collection of eigenfunctions with respect to each operator that are orthogonal with respect to the natural measures. This is contained in the next theorem
(see also \cite{expanders} for part $(i)$). Recall the notation in Theorems \ref{inverHer}, \ref{inverLag} and \ref{inverJac}.

\begin{theorem}[Orthogonal bases]\label{inverAll}~
\begin{enumerate}[$(i)$]
\item Let $\{H_n\}_{n\geq0}$ be the family of classical Hermite polynomials given by the Rodrigues formula
$$H_n(x) = (-1)^ne^{x^2} \frac{d^ne^{-x^2}}{dx^n},\quad x\in\R.$$
Then the functions $H^*_n(x) = e^{-x^2} H_n(x)$ are orthogonal with respect to the inverse Gaussian measure
$d\gamma_{-1}(x)= \pi^{1/2} e^{x^2}dx$ on $\R$.
Moreover, $\widetilde{\mathcal O} H^*_n = -(2n+2) H^*_n$ and the family  $\{H^*_n\}_{n\geq0}$ is an orthogonal basis of $L^2(\gamma_{-1})$.
\item Fix $\alpha>-1$. Let $\{L_n^\alpha\}_{n\geq0}$ be the family of classical Laguerre polynomials given by the Rodrigues formula
$$L_n^\alpha(x)=  \frac{e^{x}x^{-\alpha}}{n!} \frac{d^n}{dx^n}(e^{-x} x^{n+\alpha}),\quad x\in(0,\infty).$$
Then the functions $L^{\alpha,*}_n(x) = e^{-x} x^{\alpha}L_k^\alpha(x)$ are orthogonal with respect to the
inverse Laguerre measure $e^{x}x^{-\alpha}dx$ on $(0,\infty)$. Moreover, $\widetilde{\mathcal L}_\alpha L^{\alpha,*}_n = -(n +\frac12) L^{\alpha,*}_n$ and the family  $\{L^{\alpha,*}_n\}_{n\geq0}$ is an 
orthogonal basis of $L^2((0,\infty),e^x x^{-\alpha} dx)$.
\item Fix $\alpha>-1, \beta>-1$. Let $\{P_n^{\alpha,\beta}\}_{n\geq0}$ the family of classical Jacobi polynomials given by the Rodrigues formula
$$P_n^{(\alpha,\beta)} (x)
=\frac{(-1)^{n+1}}{2^nn!}(1-x)^{-\alpha} (1+x)^{-\beta}\frac{d^n}{dx^n}
 \big( (1-x)^{\alpha+n}(1+x)^{\beta+n}\big),\quad x\in (-1,1).$$
 Then the functions 
$P_n^{(\alpha,\beta),*} (x)= (1-x)^{\alpha}(1+x)^{\beta} P_n^{(\alpha,\beta)}(x)$ are orthogonal
with respect to the measure $ (1-x)^{-\alpha}(1+x)^{-\beta}dx$ on $(-1,1)$. Moreover,
$\widetilde{\mathcal G}_{\alpha,\beta} P_n^{(\alpha,\beta),*}= (-n(n+\alpha+\beta+1) -2)P_n^{(\alpha,\beta),*}$ and  the family  $\{P_n^{(\alpha,\beta),*}\}_{n\geq0}$ is an orthogonal basis of $L^2((0,1), (1-x)^{-\alpha}(1+x)^{-\beta}dx)$.
\end{enumerate}
\end{theorem}

Even though the families of polynomials in Theorems \ref{inverHer}, \ref{inverLag} and \ref{inverJac} are not orthogonal with respect to their corresponding inverse measures, due to the recurrence formulas we discovered, we can establish a Favard-type theorem.

\begin{theorem}[Favard-type theorem]\label{favard}
Let $\{\widetilde{Q} _n\}_{n\geq0}$ be the monic version of the sequence of polynomials
defined either in Theorems \ref{inverHer}, \ref{inverLag}, or \ref{inverJac},
and let $\lambda_1$ be their first eigenvalue with respect to their associated second order differential operator.
Then, for each case, there is a unique moment functional $\mathcal{L}$ such that 
$$\mathcal{L}[1] = \lambda_1, \quad \mathcal{L}[\widetilde{Q}_m(x) \widetilde{Q}_n(x)] = 0,
\quad \mathcal{L}[\widetilde{Q}_n^2(x) ] \neq 0,~\hbox{for}~m,n\geq0,~m\neq n.$$
\end{theorem}
  
As an application of the ideas we developed in this paper, 
we prove boundedness results for singular integral operators in inverse measure settings.
We begin with the inverse Gaussian measure in $\R^d$, $d\geq1$.

\begin{theorem}\label{aplicaGauss}
The following operators are bounded on $L^2(\gamma_{-d})=L^2(\R^d,\pi^{d/2}e^{|x|^2}dx)$.
\begin{itemize}
\item[a)] Maximal semigroup operator:  $\mathcal{M} f(x) =\sup_{t>0} |e^{-t \widetilde{\mathcal{O}} } f(x)|$;
\item[b)] Riesz transforms: $R_i f(x) = \partial_{x_i} \widetilde{\mathcal{O}}^{-1/2} f(x) $, 
$R^*_if(x) = (\partial_{x_i} +2x_i ) \widetilde{\mathcal{O}}^{-1/2} f(x)$, for $i=1,\dots, d $; and  
\item[c)] Littlewood--Paley square function: $\displaystyle{G}f(x) = \Big(\int_0^\infty | \partial_t e^{-t \widetilde{\mathcal{O}}} f(x)|^2 \frac{dt}{t} \Big)^{1/2}. $
\end{itemize}
\end{theorem}

Similarly, we have the following results for inverse Laguerre and Jacobi measures.
  
  \begin{theorem}\label{aplicaLaguerre}
  Recall the notation in Theorems \ref{inverLag} and \ref{inverJac}.
  \begin{itemize}
  \item[(A)] The following operators are bounded on $L^2((0,\infty),x^{-\alpha} e^{x} dx)$, $\alpha>-1$.
\begin{itemize}
\item[a.1)] Maximal semigroup operator:  $\mathcal{M}_\alpha f(x) =\sup_{t>0} |e^{-t \widetilde{\mathfrak{L}}_\alpha } f(x)|$; and
\item[a.2)] Littlewood--Paley square function: $\displaystyle{{G}}_\alpha f(x) = \Big(\int_0^\infty | \partial_t e^{-t \widetilde{\mathfrak{L}}_\alpha} f(x)|^2 \frac{dt}{t} \Big)^{1/2}. $
\end{itemize}
\item[(B)] The following operators are bounded on $L^2((-1,1),(1-x)^{-\alpha} (1+x)^{-\beta}dx), \alpha, \beta >0$.
\begin{itemize}
\item[b.1)] Maximal semigroup operator: $\mathcal{M}_{\alpha, \beta} f(x) =\sup_{t>0} |e^{-t \widetilde{\mathcal{G}}_{\alpha,\beta}} f(x)|$; and
\item[b.2)] Littlewood--Paley square function: $\displaystyle{{G}}_{\alpha,\beta} f(x) = \Big(\int_0^\infty | \partial_t e^{-t \widetilde{\mathcal{G}}_{\alpha,\beta}} f(x)|^2 \frac{dt}{t} \Big)^{1/2}. $
\end{itemize}
\end{itemize}
\end{theorem}

In Section \ref{part:fractional} we establish the theory of fractional powers of first order operators
and, in Section \ref{part:orthogonal}, we present the new inverse measures and polynomials, prove the results above,
and provide further comments. The reader should not confuse the parameter $\alpha$ from Section \ref{part:fractional}
(the fractional power of the first order operators) with the parameter $\alpha>-1$ from Section \ref{part:orthogonal}
(that denotes the type of Laguerre or Jacobi polynomials). We allow this apparent abuse of notation to be consistent with
the usual notations in both the classical theory of fractional derivatives and orthogonal polynomials literatures.

%%%%%%%%%%%%%%%%%%%%%%%%%%%%%%%%%%%%%%%%%%%%%%%%%%%%%%
\section{Fractional powers of first order differential operators}\label{part:fractional}
%%%%%%%%%%%%%%%%%%%%%%%%%%%%%%%%%%%%%%%%%%%%%%%%%%%%%%

Throughout this Section, we let $a=a(x):\R\to\R$ be a continuous function.

%%%%%%%%%%%%%%%%%%%%%%%%%%%%%%%%%%%%%%%%%%%%%%%%%%%%%%
\subsection{First order differential operators}
%%%%%%%%%%%%%%%%%%%%%%%%%%%%%%%%%%%%%%%%%%%%%%%%%%%%%%

\begin{definition}
Given a measurable function $u=u(x):\R\to\R$ we define the operators
\begin{equation}\label{definioperadores}
\begin{aligned}
\mathfrak{D}_{{\rm left},a}u(x)&=D_{{\rm left}}u(x) +a(x)u(x)\\
\mathfrak{D}_{{\rm right},a}u(x)&=D_{{\rm right}}u(x) +a(x)u(x),
\end{aligned}
\end{equation}
for $x\in\R$, where $D_{{\rm left}}$ and $D_{{\rm right}}$ are the \emph{derivatives from the left} and \emph{from the right}, respectively,
which were defined in \cite[Definition~2.3]{BerMarStiTor} as
\begin{equation}\label{izda}
\begin{aligned}
D_{{\rm left}}u(x) = \lim_{t\rightarrow 0^+}\frac{u(x)- u(x-t)}{t}\\
D_{{\rm right}}u(x) = \lim_{t\rightarrow 0^+} \frac{u(x)-u(x+t)}{t},
\end{aligned}
\end{equation}
whenever the limits exist.
\end{definition}

Observe that $D_{{\rm right}}$ equals the negative of the right
lateral derivative $\frac{d}{dx^{+}}$ as usually defined in calculus,
while $D_{{\rm left}}$ coincides with the usual left lateral derivative $\frac{d}{dx^{-}}$.
Thus, if $u$ is differentiable, then $D_{\mathrm{left}}u=-D_{\mathrm{right}}u=u'$.

\begin{remark}\label{rem:other}
\emph{Consider the first order operator
$$b(x)u'+a(x)u.$$
Suppose that $b(x)$ and $u$ are differentiable.
If $b(x)$ does not vanish at any point, then
$$b(x)u'+a(x)u=(b(x)u)'-b'(x)u+a(x)u=(b(x)u)'+\frac{a(x)-b'(x)}{b(x)}(b(x)u).$$
Thus, if $b(x)$ does not change sign, then
\begin{align*}
b(x)\Dl u(x)+a(x)u(x) &=\mathfrak{D}_{{\rm left},[(a-\Dl b)/b]}(bu)(x)\\
b(x)\Dr u(x)+a(x)u(x) &=\mathfrak{D}_{{\rm right},[(a-\Dr b)/b]}(bu)(x).
\end{align*}
Hence, all the results we will present here for the case when $b\equiv1$ will still be valid for a variable coefficient $b(x)$
as soon as either $b(x)>0$ or $b(x)<0$, for all $x\in\R$.}
\end{remark}

It is clear that if $u$ and $v$ are compactly supported functions on $\R$ such that $u$ is left differentiable and
$v$ is right differentiable, then 
$$\int_\mathbb{R}(\mathfrak{D}_{{\rm left},a}u)v\,dx = \int_\mathbb{R} u(\mathfrak{D}_{{\rm right},a}v)\,dx.$$

%%%%%%%%%%%%%%%%%%%%%%%%%%%%%%%%%%%%%%%%%%%%%%%%%%%%%%
\subsection{Semigroup analysis}
%%%%%%%%%%%%%%%%%%%%%%%%%%%%%%%%%%%%%%%%%%%%%%%%%%%%%%

Recall that the semigroups generated by $D_{\mathrm{left}}$ and $D_{\mathrm{right}}$
acting on functions $u$ and $v$ are given by 
\begin{equation}\label{eq:translationsemigroups}
\begin{aligned}
e^{-tD_{{\rm left}}}u(x)&=u(x-t)\\
e^{-tD_{{\rm right}}}v(x)&=v(x+t)
\end{aligned}
\end{equation}
for $x\in\R$, $t\geq0$, respectively, see \cite{BerMarStiTor}.
Clearly, if $u,v\in L^p(\R)$ then $\|e^{-tD_{{\rm left}}}u\|_{L^p(\R)}=\|u\|_{L^p(\R)}$
and $\|e^{-tD_{{\rm right}}}v\|_{L^p(\R)}=\|v\|_{L^p(\R)}$, for all $t\geq0$ and $1\leq p\leq\infty$.

Given any $x_0 \in \mathbb{R}$, we introduce the positive function
\begin{equation}\label{definicion}
\E(x)=\E_{a,x_0}(x)=\exp\bigg[-\int_{x_0}^x a(y)\,dy\bigg]\qquad\hbox{for}~x\in\R.
\end{equation}
Notice that $\E_{a,x_0}(x)^{-1}=1/\E_{a,x_0}(x)=\E_{-a,x_0}(x)$, $\E'(x)=-a(x)\E(x)$ and $[\E(x)^{-1}]'=a(x)\E(x)^{-1}$.

\begin{theorem}[Semigroup analysis]\label{dualidadL}
Let  $u,v:\R\to\R$ be measurable functions.
\begin{enumerate}[$(i)$]
\item If we define, for any $x\in\R$ and $t\geq 0$, the linear operators
\begin{align*}
e^{-t\mathfrak{D}_{{\rm left},a}}u(x)&=\exp\bigg[-\int_{x-t}^{x}a(y)\,dy\bigg]u(x-t) \\
e^{-t\mathfrak{D}_{{\rm right},a}}v(x)&=\exp\bigg[-\int_{x}^{x+t}a(y)\,dy\bigg]v(x+t),
\end{align*}
then $e^{-0\mathfrak{D}_{{\rm left},a}}u(x)=u(x)$ and $e^{-0\mathfrak{D}_{{\rm right},a}}v(x)=v(x)$ a.e.~and, for all $t\geq0$,
\begin{align*}
e^{-t\mathfrak{D}_{{\rm left},a}}u(x) &= \E(x)e^{-tD_{{\rm left}}}(\E^{-1}u)(x) \\
e^{-t\mathfrak{D}_{{\rm right},a}}v(x) &= \E(x)^{-1}e^{-tD_{{\rm right}}}(\E v)(x).
\end{align*}
Moreover, if $u$ and $v$ have compact support, then
$$\int_\mathbb{R}\big(e^{-t\mathfrak{D}_{{\rm left},a}}u\big)v\,dx = \int_\mathbb{R}u\big(e^{-t\mathfrak{D}_{{\rm right},a}}v\big)\,dx.$$
\item For any $t_1,t_2\geq0$ and a.e.~$x\in\R$,
$$e^{-t_1\mathfrak{D}_{{\rm left},a}}(e^{-t_2\mathfrak{D}_{{\rm left},a}}u)(x)=e^{-(t_1+t_2)\mathfrak{D}_{{\rm left},a}}u(x)$$
$$e^{-t_1\mathfrak{D}_{{\rm right},a}}(e^{-t_2\mathfrak{D}_{{\rm right},a}}v)(x)=e^{-(t_1+t_2)\mathfrak{D}_{{\rm right},a}}v(x).$$
Moreover, if $u$ is left differentiable and $v$ is right differentiable at $x\in\R$, then
\begin{align*}
\lim_{t\to0^+}\frac{e^{-t\mathfrak{D}_{{\rm left},a}}u(x)-u(x)}{t}&=-\mathfrak{D}_{{\rm left},a}u(x)\\
\lim_{t\to0^+}\frac{e^{-t\mathfrak{D}_{{\rm right},a}}v(x)-v(x)}{t}&=-\mathfrak{D}_{{\rm right},a}v(x).
\end{align*}
In other words, $\mathfrak{D}_{{\rm left},a}$ and $\mathfrak{D}_{{\rm right},a}$, as defined in \eqref{definioperadores},
are the infinitesimal generators of the semigroups $\{e^{-t\mathfrak{D}_{{\rm left},a}}\}_{t\geq 0}$ and
$\{e^{-t\mathfrak{D}_{{\rm right},a}}\}_{t\geq 0}$, respectively.
\item For each $t\geq0$ and any $1\leq p\leq\infty$, if $u\in L^p(\E^{-p}dx)$ and $v\in L^p(\E^{p}dx)$ then
$$\|e^{-t \mathfrak{D}_{{\rm left},a}}u\|_{L^p(\E^{-p}dx)}=\|u\|_{L^p(\E^{-p}dx)}$$
$$\|e^{-t \mathfrak{D}_{{\rm right},a}}v\|_{L^p(\E^{p}dx)}=\|v\|_{L^p(\E^{p}dx)}.$$
\end{enumerate}
\end{theorem}

\begin{proof}
For $(i)$, by using \eqref{definicion},
\begin{align*}
e^{-t\mathfrak{D}_{{\rm left},a}}u(x) &= \exp\bigg[-\int_{x-t}^{x}a(y)\,dy\bigg]u(x-t) \\
&= \E(x)\exp\bigg[\int_{x_0}^{x}a(y)\,dy-\int_{x-t}^{x}a(y)\,dy\bigg]u(x-t) \\
&= \E(x)\E(x-t)^{-1}u(x-t)=\E(x)e^{-tD_{{\rm left}}}(\E^{-1}u)(x).
\end{align*}
The statements in $(ii)$ are easy to check by using \eqref{izda} and the fundamental theorem of calculus.
For $(iii)$, observe that, by $(i)$, if $u\in L^p(\E^{-p}dx)$, $1\leq p\leq\infty$ then
$$\|e^{-t \mathfrak{D}_{{\rm left},a}}u\|_{L^p(\E^{-p}dx)}= 
\|e^{-t D_{\rm left}}(\E^{-1}u)\|_{L^p(\R)}= 
\|\E^{-1}u\|_{L^p(\R)}= \|u\|_{L^p(\E^{-p}dx)}.$$

All the statements for $e^{-t \mathfrak{D}_{{\rm right},a}}$ follow similar lines.
\end{proof}

%%%%%%%%%%%%%%%%%%%%%%%%%%%%%%%%%%%%%%%%%%%%%%%%%%%%%%
\subsection{Fractional positive powers of first order differential operators}\label{subsect:fractionalpositive}
%%%%%%%%%%%%%%%%%%%%%%%%%%%%%%%%%%%%%%%%%%%%%%%%%%%%%%

\begin{definition}
For $0<\alpha<1$, we define the positive fractional power operators
$$(\mathfrak{D}_{{\rm left},a})^\alpha u(x)
= \frac1{\Gamma(-\alpha)} \int_0^\infty\big(e^{-t\mathfrak{D}_{{\rm left},a}}u(x)-u(x)\big)\,\frac{dt}{t^{1+\alpha}}$$
and
$$(\mathfrak{D}_{{\rm right},a})^\alpha v(x)
= \frac1{\Gamma(-\alpha)} \int_0^\infty\big(e^{-t\mathfrak{D}_{{\rm right},a}}v(x)-v(x)\big)\,\frac{dt}{t^{1+\alpha}}$$
whenever the integrals exist.
\end{definition}

We recall that in \cite{BerMarStiTor} the positive powers of the left and right derivative operators \eqref{izda}
were defined by using the corresponding
semigroups \eqref{eq:translationsemigroups}, for $0<\alpha<1$, as
\begin{align*}
(\Dl)^\alpha u(x) &= \frac1{\Gamma(-\alpha)} \int_0^\infty\big(e^{-t\Dl}u(x)-u(x)\big)\,\frac{dt}{t^{1+\alpha}}\\
&=c_\alpha \int_{-\infty}^x\frac{u(x)-u(t)}{(x-t)^{1+\alpha}}\,dt
\end{align*}
and
\begin{align*}
(\Dr)^\alpha v(x) &= \frac1{\Gamma(-\alpha)} \int_0^\infty\big(e^{-t\Dr}v(x)-v(x)\big)\,\frac{dt}{t^{1+\alpha}}\\
&=c_\alpha \int_x^{\infty}\frac{v(x)-v(t)}{(t-x)^{1+\alpha}}\,dt,
\end{align*}
where $c_{\alpha}=1/|\Gamma(-\alpha)|>0$.
Here $u$ is a sufficiently regular function, say, in the Schwartz class $\mathcal{S}(\R)$.
Then $(\Dl)^\alpha$ coincides with the left-sided Marchaud fractional derivative of order $0<\alpha<1$.
It can then be checked that if $u,v\in\mathcal{S}(\R)$ then
\begin{equation}\label{dualidadDe}
\int_\mathbb{R}[(D_{\rm left})^{\alpha} u(x)]v(x)\,dx = \int_\mathbb{R} u(x)[(D_{\rm right})^{\alpha} v(x)]\,dx.
\end{equation}

As a consequence of Theorem \ref{dualidadL}$(i)$, the semigroup definitions of
$(\Dl)^{\alpha}$ and $(\Dr)^{\alpha}$ and \eqref{dualidadDe}, we get the following fundamental identities.

\begin{lemma}\label{lem:fractionalconjugation}
Let $u$ and $v$ be functions such that $\E^{-1}u,\E v\in\mathcal{S}(\R)$. Then, for all $0<\alpha<1$,
\begin{align*}
(\mathfrak{D}_{{\rm left},a})^{\alpha}u(x) &= \E(x)(\Dl)^{\alpha}(\E^{-1}u)(x) \\
&=c_\alpha \E(x)\int_{-\infty}^x\frac{(\E^{-1}u)(x)-(\E^{-1}u)(t)}{(x-t)^{1+\alpha}}\,dt
\end{align*}
and
\begin{align*}
(\mathfrak{D}_{{\rm right},a})^{\alpha}v(x) &= \E(x)^{-1}(\Dr)^{\alpha}(\E v)(x) \\
&=c_\alpha \E^{-1}(x)\int_x^{\infty} \frac{(\E v)(x)-(\E v)(t)}{(t-x)^{1+\alpha}}\,dt.
\end{align*}
In addition,
$$\int_\R[(\mathfrak{D}_{{\rm left},a})^\alpha u(x)]v(x)\,dx=\int_\R u(x)[(\mathfrak{D}_{{\rm right},a})^\alpha v(x)]\,dx.$$
\end{lemma}

Observe that the case $\alpha=1$ in the previous lemma takes the following form.

\begin{lem}\label{lem:conjugationalpha1}
Let $x\in\R$. If $u$ is left differentiable at $x$ and $v$ is right differentiable at $x$, then 
\begin{align*}
\mathfrak{D}_{{\rm left},a}u(x)&=\E(x)D_{\rm left}(\E^{-1}u)(x)\\
\mathfrak{D}_{{\rm right},a}v(x)&=\E(x)^{-1}D_{\rm right}(\E v)(x).
\end{align*}
\end{lem}

\begin{rem}
\emph{A more precise notation for the fractional operators above would be
$$(\mathfrak{D}_{{\rm left},a,x_0})^\alpha\qquad\hbox{and}\qquad(\mathfrak{D}_{{\rm right},a,x_0})^\alpha,$$
where the dependence on the a priori chosen point $x_0$ is made explicit.
Define the reflection about the origin of a function $u$
as  $\widetilde{u}(x):=u(-x)$. Observe that
$$\E_{\widetilde{a},-x_0}(-x)=(\E_{a,x_0})^{-1}(x)\qquad\hbox{and}\qquad \E_{\widetilde{a},-x_0}^{-1}(-x)=\E_{a,x_0}(x).$$
Using these relations, it is readily checked that
$$(\mathfrak{D}_{{\rm right},a,x_0})^\alpha u(x)=[(\mathfrak{D}_{{\rm left},\widetilde{a},-x_0})^\alpha\widetilde{u}](-x).$$
Hence, from now on we will only work with $(\mathfrak{D}_{{\rm left},a,x_0})^\alpha$ as parallel results for
$(\mathfrak{D}_{{\rm right},a,x_0})^\alpha$ can be found by reversing the orientation of the real line.}
\end{rem}

%%%%%%%%%%%%%%%%%%%%%%%%%%%%%%%%%%%%%%%%%%%%%%%%%%%%%%
\subsection{Distributional definition}
%%%%%%%%%%%%%%%%%%%%%%%%%%%%%%%%%%%%%%%%%%%%%%%%%%%%%%

Precise conditions on $u$ so that the fractional operators and identities in
Lemma \ref{lem:fractionalconjugation} are well defined are given next. 
In particular, we extend the definition of the fractional operator
$(\mathfrak{D}_{{\rm left},a})^{\alpha}$ to distributions and, in particular, to functions in $L^p$.
It has been observed in \cite{BerMarStiTor} that $(\Dl)^\alpha$ and $(\Dr)^\alpha$ do not preserve
the Schwartz class $\mathcal{S}(\mathbb{R})$.
Several attempts have been made to find an appropriate space of distributions to define these operators,
see, for instance, \cite{BerMarStiTor,mary}. We shall present a quick overview of the strategy in \cite{mary}
for the left-sided fractional derivative.

Define the left-sided class of test functions
$$\mathcal{S}_-=\big\{ \varphi \in \mathcal{S}(\R): \supp\,\varphi\subset(-\infty,A],~\hbox{for some}~A\in \mathbb{R} \big\}$$
endowed with the family of seminorms $ \sup_{x\in \mathbb{R}}|x|^\ell\big| \frac{d^k}{d x^k}\varphi(x)\big|$,
for $\ell,k\geq0$. 
It is shown in \cite[Lemma~2.1]{mary} that if $\varphi \in \mathcal{S}_-$ then $(\Dr)^\alpha \varphi \in \S_-^\alpha$, where
\begin{align*}
\S_-^\alpha  &=  \Big\{\varphi \in C^\infty(\mathbb{R}): \supp\,\varphi \subset (-\infty,A]~\hbox{for some}~A\in\R,~\hbox{and}\\
&\qquad\qquad\qquad\qquad\Big| \frac{d^k \varphi(x)}{dx^k}\Big|
\le \frac{C_k}{1+|x|^{1+\alpha}},~\hbox{for all}~k \ge 0~\hbox{and some}~C_k>0\Big\},
\end{align*}
and this class is endowed with the family of seminorms
$ \sup_{x\in \mathbb{R}}(1+|x|^{1+\alpha})\Big| \frac{d^k}{dx^k}\varphi(x)\Big|$, for $k\geq0$. 
Thus, as a consequence of the duality identity \eqref{dualidadDe}, the left fractional derivative
$$(\Dl)^\alpha:(\S_-^\alpha)'\to(\S_-)'$$
continuously via
$$[(\Dl)^\alpha u](\varphi)=u((\Dr)^\alpha\varphi)\qquad\hbox{for}~\varphi\in\S_-,$$
see \cite[Definition~2.2]{mary}.
In addition, the class of functions
$$L^\alpha_-:=\bigg\{u\in L^1_{\mathrm{loc}}(\R):\|u\|_{A}:=\int_{-\infty}^A\frac{|u(x)|}{1+|x|^{1+\alpha}}\,dx< \infty~\hbox{for any}~A\in\R\bigg\}$$
is a subspace of $(\S_-^\alpha)'$. Consequently, $(\Dl)^\alpha$ is also well-defined over this space of functions.

An important subclass of functions in $L^\alpha_-$ is the one-sided $L^p(\omega)$ space, for $1\leq p<\infty$.
A nonnegative, locally integrable function $\omega=\omega(x)$ defined on $\R$ is in the left-sided Sawyer class $A_p^-$,
$1<p<\infty$, if there is a constant $C>0$ such that
$$\bigg(\frac{1}{h}\int_x^{x+h}\omega(y)\,dy\bigg)^{1/p}\bigg(\frac{1}{h}\int_{x-h}^x\omega(y)^{1-p'}\,dy\bigg)^{1/p'}\leq C$$
for all $x\in\R$ and $h>0$, where $1/p+1/p'=1$. Without loss of generality, it can be assumed that $\omega>0$.
The one-sided Hardy--Littlewood maximal function is given by
$$M^+u(x)=\sup_{h>0}\int_{x}^{x+h}|u(y)|\,dy.$$
A weight $\omega$ belongs to the class $A_1^-$ if there exists a constant $C>0$ such that $M^+\omega\leq C\omega$ almost everywhere in $\R$.
It is proved in \cite[Proposition~2.8]{mary} that $L^p(\omega)\subset L^\alpha_-$, for any $\omega\in A^p_-$, $1\leq p<\infty$ and $\alpha>0$.

Another important class of weights we will use are those in the class $A_{p,q}^-$, for $1<p,q<\infty$.
We say that $\omega\in A_{p,q}^-$ if there exists a constant $C>0$ such that
$$\bigg(\frac{1}{h} \int_{x}^{x+h}\omega(y)^q\,dy\bigg)^{1/q}\bigg(\frac{1}{h} \int_{x-h}^x\omega(y)^{-p'}\,dy\bigg)^{1/p'} \leq C$$
for all $x\in\R$ and all $h>0$. Notice that $\omega\in A_{p,q}^-$ if and only if $\omega^q \in A_r^-$ for $r = 1+ \frac{q}{p'}$.

Define next the following spaces of test functions:
$$\mathcal{S}^\alpha_{-,\E}=\big\{\psi\in C^\infty(\R):\E\psi\in\mathcal{S}_-^\alpha\big\}$$
under the family of seminorms $ \sup_{x\in \mathbb{R}}(1+|x|^{1+\alpha})\Big| \frac{d^k}{dx^k}(\E\psi)(x)\Big|$, for $k\geq0$;
and
$$\mathcal{S}_{-,\E}=\big\{\psi\in C^\infty(\R):\E\psi\in\mathcal{S}_-\big\}$$
under the family of seminorms $ \sup_{x\in \mathbb{R}}|x|^\ell\big| \frac{d^k}{d x^k}(\E\psi)(x)\big|$,
for $\ell,k\geq0$. 

\begin{lem}\label{lem:distributionaldefinition}
Given a distribution $u$ in the dual space $(\mathcal{S}^\alpha_{-,\E})'$ we define $\mathfrak{D}_{{\rm left},a}u$
as a distribution in $(\mathcal{S}_{-,\E})'$ via
$$[\mathfrak{D}_{{\rm left},a} u](\psi)=u[\E^{-1}(\Dr)^\alpha(\E\psi)]\qquad\hbox{for}~\psi\in\mathcal{S}_{-,\E}.$$
Then $\mathfrak{D}_{{\rm left},a}:(\mathcal{S}^\alpha_{-,\E})'\to(\mathcal{S}_{-,\E})'$ is continuous.
In particular, if $\E^{-1}u\in L^p(\omega)$, where $\omega\in A_p^-$, $1\leq p<\infty$, then
$\mathfrak{D}_{{\rm left},a}u$ is well defined as a distribution in $(\mathcal{S}_{-,\E})'$.

Furthermore, if $u$ is a measurable function
such that $\E^{-1}u\in L^\alpha_-$ and $\E^{-1}u\in C^{\alpha+\varepsilon}(I)$ in some interval $I$,
for some $\varepsilon>0$, then $\mathfrak{D}_{{\rm left},a} u$ is well-defined
and it coincides with a continuous function on $I$. Moreover,
\begin{equation}\label{eq:pointwiseformulaDleft}
(\mathfrak{D}_{{\rm left},a})^\alpha u(x)=c_\alpha \E(x)\int_{-\infty}^x
\frac{(\E^{-1}u)(x)-(\E^{-1}u)(t)}{(x-t)^{1+\alpha}}\,dt\qquad\hbox{for all}~x\in I.
\end{equation}
\end{lem}

\begin{proof}
Let $\psi\in\mathcal{S}_{-,\E}$. Then $\E\psi\in\mathcal{S}_-$ and so $(\Dr)^\alpha(\E\psi)\in\mathcal{S}_-^\alpha$.
It then follows that $\E^{-1}(\Dr)^\alpha(\E\psi)\in\mathcal{S}_{-,\E}^\alpha$ and, thus, $u[\E^{-1}(\Dr)^\alpha(\E\psi)]$
is well-defined. The continuity of $\mathfrak{D}_{{\rm left},a}$ follows from the continuity of $u$ and $(\Dr)^\alpha$.

Observe that $u\in(\S_{-,\E}^\alpha)'$ if and only if $\E^{-1}u\in(\S_-^\alpha)'$, where the latter distribution is defined
as $\E^{-1}u(\varphi)=u(\E^{-1}\varphi)$, for all $\varphi\in\S_-^\alpha$. Hence, if $u$ is a measurable function
such that $\E^{-1}u\in L^\alpha_-\subset(\S_-^\alpha)'$, then $(\Dl)^\alpha(\E^{-1}u)$ is well-defined as a distribution
in $(\S_-)'$ and so $\E(\Dl)^\alpha(\E^{-1}u)\in(\S_{-,\E})'$.
Furthermore, if $\E^{-1}u\in C^{\alpha+\varepsilon}(I)$, by a standard approximation argument (see \cite{mary}),
it follows that $\E^{-1}(\Dl)^\alpha(\E^{-1}u)$ coincides with
the pointwise formula in the statement, for all $x\in I$. Moreover, $\E^{-1}(\Dl)^\alpha(\E^{-1}u)$
is a continuous function on $I$.
By the distributional definition of $\mathfrak{D}_{{\rm left},a}$, 
such pointwise formula coincides with $\mathfrak{D}_{{\rm left},a}u$ in $I$.

Finally, if $\E^{-1}u\in L^p(\omega)$, $\omega\in A^p_-$, $1\leq p<\infty$, then $\E^{-1}u\in L^\alpha_-\subset(\mathcal{S}_-^\alpha)'$ and
the analysis above shows that $\mathfrak{D}_{{\rm left},a}u\in(\mathcal{S}_{-,\E})'$.
\end{proof}

%%%%%%%%%%%%%%%%%%%%%%%%%%%%%%%%%%%%%%%%%%%%%%%%%%%%%%
\subsection{Maximum principles}
%%%%%%%%%%%%%%%%%%%%%%%%%%%%%%%%%%%%%%%%%%%%%%%%%%%%%%

Recall from the previous discussion that if $u$ is a measurable function
such that $\E^{-1}u\in L_-^\alpha\cap C^{\alpha+\varepsilon}(I)$ in some interval open $I\subset\R$,
then $(\mathfrak{D}_{{\rm left},a})^\alpha u(x)$
is a continuous function of $I$ given by the pointwise formula \eqref{eq:pointwiseformulaDleft}.

\begin{theorem}[Maximum principles]\label{theorem:A}
Let $0<\alpha<1$ and $\E(x)$ be as in \eqref{definicion}.
\begin{enumerate}[$(a)$]
\item Let $u$ be a measurable function such that $\E^{-1}u\in L^\alpha_-$. Assume that there is a point $x_0$ such that
$u(x_0)=0$, $u(x)\geq0$ for a.e.~$x\leq x_0$, and $(\mathfrak{D}_{{\rm left},a})^\alpha u(x_0)$
is given by the pointwise formula \eqref{eq:pointwiseformulaDleft}. Then $(\mathfrak{D}_{{\rm left},a})^\alpha u(x_0)\leq0$. Moreover,
$(\mathfrak{D}_{{\rm left},a})^\alpha u(x_0) = 0$ if and only if $u(x)=0$ for a.e.~$x\leq x_0$.
\item Let $u,v$ be measurable functions such that $\E^{-1}u,\E^{-1}v\in L^\alpha_-$
Assume that there is a point $x_0$ such that $u(x_0)=v(x_0)$, $u(x) \ge v(x)$ for a.e.~$x \leq x_0$, and 
$(\mathfrak{D}_{{\rm left},a})^\alpha u(x_0)$ and $(\mathfrak{D}_{{\rm left},a})^\alpha v(x_0)$
are given by the pointwise formula \eqref{eq:pointwiseformulaDleft}.
Then  $(\mathfrak{D}_{{\rm left},a})^\alpha u(x_0) \le (\mathfrak{D}_{{\rm left},a})^\alpha v(x_0)$.
Moreover, $(\mathfrak{D}_{{\rm left},a})^\alpha u(x_0) = (\mathfrak{D}_{{\rm left},a})^\alpha v(x_0)$
if and only if $u(x)=v(x)$ for a.e.~$x \le x_0$.
\end{enumerate}
\end{theorem}

\begin{proof}
Part $(b)$ follows from $(a)$ by considering $u-v$ and using the linearity of $(\mathfrak{D}_{{\rm left},a})^\alpha$.
Let $u$ be as in $(a)$. Then
$$(\mathfrak{D}_{{\rm left},a})^\alpha u(x_0)=-c_\alpha \E(x_0)\int_{-\infty}^{x_0}\frac{(\E^{-1}u)(t)}{(x_0-t)^{1+\alpha}}\,dt\leq0$$
because $\E>0$. Moreover, since the integrand is nonnegative, $(\mathfrak{D}_{{\rm left},a})^\alpha u(x_0)=0$ if and only if $u(x)=0$
for a.e.~$x\leq x_0$.
\end{proof}

\begin{cor}\label{cor:1}
Let $u,v$ be such that $\E^{-1}u,\E^{-1}v\in L^\alpha_-$, and $T>0$.
Assume that $u$ and $v$ are continuous on $(-\infty,T]$ and that
$(\mathfrak{D}_{{\rm left},a})^\alpha u(x)$ and $(\mathfrak{D}_{{\rm left},a})^\alpha v(x)$
are given by the pointwise formulas \eqref{eq:pointwiseformulaDleft}, whenever $x\in(0,T]$.
\begin{enumerate}[$(i)$]
\item If
$$\begin{cases}
(\mathfrak{D}_{{\rm left},a})^\alpha u\leq0&\hbox{in}~(0,T]\\
u\leq0&\hbox{in}~(-\infty,0],
\end{cases}$$
then $\sup_{x<T}u(x)=\sup_{x<0}u(x)$.
\item If
$$\begin{cases}
(\mathfrak{D}_{{\rm left},a})^\alpha u\geq0&\hbox{in}~(0,T]\\
u\geq0&\hbox{in}~(-\infty,0],
\end{cases}$$
then $\inf_{x<T}u(x)=\inf_{x<0}u(x)$.
\item If
$$\begin{cases}
(\mathfrak{D}_{{\rm left},a})^\alpha u\geq(\mathfrak{D}_{{\rm left},a})^\alpha v&\hbox{in}~(0,T]\\
u\geq v&\hbox{in}~(-\infty,0],
\end{cases}$$
then $u\geq v$ in $(-\infty,T]$. In particular, we have uniqueness for the Dirichlet problem
$$\begin{cases}
(\mathfrak{D}_{{\rm right},a})^\alpha u=f&\hbox{in}~(0,T]\\
u=g&\hbox{in}~(-\infty,0].
\end{cases}$$
\end{enumerate}
\end{cor}

\begin{proof}
Notice that $(i)$ is equivalent to $(ii)$ by considering $-u$, and $(iii)$ is a consequence of $(ii)$ applied to $u-v$.
We are left to prove $(i)$. Assume, by way of contradiction, that $\sup_{x<T}u(x)$ is not attained in $(-\infty,0]$. Since $u$ is continuous,
there exists $x_0\in(0,T]$ such that $u(x_0)$ is a global maximum of $u$ in $(-\infty,T]$. Hence, by Theorem \ref{theorem:A}
applied to $u(x_0)-u$, we find that $(\mathfrak{D}_{{\rm left},a})^\alpha u(x_0)\geq0$.
Since $(\mathfrak{D}_{{\rm left},a})^\alpha u\leq0$ in $(0,T]$, we have
$(\mathfrak{D}_{{\rm left},a})^\alpha u(x_0)=0$. Thus, by Theorem \ref{theorem:A}, $u(x)=u(x_0)$ for all $x\leq x_0$,
so $\sup_{x<T}u(x)$ is attained in $(-\infty,0]$, a contradiction to our initial assumption.
\end{proof}

%%%%%%%%%%%%%%%%%%%%%%%%%%%%%%%%%%%%%%%%%%%%%%%%%%%%%%
\subsection{Extension problem}
%%%%%%%%%%%%%%%%%%%%%%%%%%%%%%%%%%%%%%%%%%%%%%%%%%%%%%

We next give a characterization of $(\mathfrak{D}_{{\rm left},a})^\alpha u$ in terms of a local PDE extension problem.

\begin{theorem}[Extension problem]\label{main}
Let $\E$ be as in \eqref{definicion} and $u\in L^p(\E^{-p}\omega)$, where $\omega\in A_p^-$, $1\leq p<\infty$.
Define the function
$$U(x,y):= \frac{y^{2\alpha}}{4^\alpha\Gamma(\alpha)}\int_0^\infty e^{-y^2/(4t)}e^{-t\mathfrak{D}_{{\rm left},a}}u(x)\,\frac{dt}{t^{1+\alpha}}$$
for $x\in\R$ and $y>0$. Then $U$ is a classical solution of the extension problem
\begin{equation}\label{extension}\begin{cases}
-\mathfrak{D}_{{\rm left},a}+\frac{1-2\alpha}{y}U_y+U_{yy}=0& \hbox{in}~\R\times(0,\infty) \\
\lim_{y\to0^+}U(x,y)=u(x)& \hbox{a.e.~and in}~L^p(\E^{-p}\omega),
\end{cases}
\end{equation}
and there exists a constant $C=C(\omega)>0$ such that
\begin{equation}\label{eq:maximalestimate}
\big\|\sup_{y\geq0}|U(\cdot,y)|\big\|_{L^p(\E^{-p}\omega)}\leq C\|u\|_{L^p(\E^{-p}\omega)}\qquad\hbox{if}~1<p<\infty
\end{equation}
and
\begin{equation}\label{eq:maximalestimate1}
\big\|\sup_{y\geq0}|U(\cdot,y)|\big\|_{\mathrm{weak}-L^1(\E^{-1}\omega)}\leq C\|u\|_{L^1(\E^{-1}\omega)}\qquad\hbox{if}~p=1.
\end{equation}
Moreover, if $c_\alpha :=\frac{4^{\a-1/2}\Gamma(\alpha)}{\Gamma(1-\a)}>0$ then
\begin{equation}\label{eq:D-to-N}
-c_\alpha\lim_{y\to0^+}y^{1-2\alpha}U_y(x,y)=(\mathfrak{D}_{{\rm left},a} )^\alpha u(x)
\end{equation}
in the sense of distributions in $(\mathcal{S}_{-,\E})'$.
\end{theorem}

\begin{proof}
Since $\E^{-1}u\in L^p(\omega)$, it follows from \cite[Theorem~1.3]{BerMarStiTor} that
$$\overline{U}(x,y):= \frac{y^{2\alpha}}{4^\alpha \Gamma(\alpha)} \int_0^\infty e^{-y^2/(4t)}e^{-tD_{\rm left}}(\E^{-1}u)(x)\,\frac{dt}{t^{1+\alpha}}$$
is a classical solution to the extension problem
$$\begin{cases}
-D_{\rm left}\overline{U}+\frac{1-2\alpha}{y}\overline{U}_y+\overline{U}_{yy}=0& \hbox{in}~\R\times(0,\infty)\\
\lim_{y\to0^+}\overline{U}(x,y)=\E^{-1}(x)u(x)& \hbox{a.e.~and in}~L^p(\omega).
\end{cases}$$
Moreover, there is a constant $C=C(\omega)>0$ such that
\begin{equation}\label{eq:maximalestimatebarra}
\big\|\sup_{y\geq0}|\overline{U}(\cdot,y)|\big\|_{L^p(\omega)}\leq\|\E^{-1}u\|_{L^p(\omega)}=\|u\|_{L^p(\E^{-p}\omega)}\qquad\hbox{if}~1<p<\infty
\end{equation}
and
\begin{equation}\label{eq:maximalestimatebarra1}
\big\|\sup_{y\geq0}|\overline{U}(\cdot,y)|\big\|_{\mathrm{weak}-L^1(\omega)}\leq C\|\E^{-1}u\|_{L^1(\omega)}
=\|u\|_{L^1(\E^{-1}\omega)}\qquad\hbox{if}~p=1,
\end{equation}
see \cite[Theorem~4.3]{BerMarStiTor}.
By Theorem \ref{dualidadL} part $(i)$, we have that $U(x,y)=\E(x)\overline{U}(x,y)$, so that, by Lemma \ref{lem:conjugationalpha1},
$U$ is a classical solution to \eqref{extension}. Also, \eqref{eq:maximalestimate} and  \eqref{eq:maximalestimate1}
follow from \eqref{eq:maximalestimatebarra} and \eqref{eq:maximalestimatebarra1}, respectively.
Moreover, it can be seen by carefully following the proof of Theorem 4.3 in \cite{BerMarStiTor} that
$$-c_\alpha\lim_{y\to0^+}y^{1-2\alpha}\overline{U}_y(x,y)=(D_{\rm left})^\alpha (\E^{-1}u)(x)$$
in the sense of distributions in $(\mathcal{S}_-)'$. Hence, by Lemma \ref{lem:distributionaldefinition},
we get \eqref{eq:D-to-N} in $(\mathcal{S}_{-,\E})'$.
\end{proof}

%%%%%%%%%%%%%%%%%%%%%%%%%%%%%%%%%%%%%%%%%%%%%%%%%%%%%%
\subsection{Sobolev spaces}
%%%%%%%%%%%%%%%%%%%%%%%%%%%%%%%%%%%%%%%%%%%%%%%%%%%%%%

For $1\leq p<\infty$ we define the (left-sided) Sobolev space
$$W^{1,p}_a(\omega)=\big\{u\in L^p(\E^{-p}\omega):\mathfrak{D}_{{\rm left},a}u\in L^p(\E^{-p}\omega)\big\}$$
where $\omega\in A_p^-$, under the natural norm
$$\|u\|_{W^{1,p}_a(\omega)}=\big(\|u\|_{L^p(\E^{-p}\omega)}^p+\|\mathfrak{D}_{{\rm left},a}u\|_{L^p(\E^{-p}\omega)}^p\big)^{1/p}.$$
Then we have the following characterization in the spirit of Bourgain--Brezis--Mironescu \cite{Bourgain-Brezis-Mironescu}.

\begin{theorem}[Characterization of Sobolev spaces]\label{thm:BBM}
Let $\E$ be as in \eqref{definicion} and $\omega\in A_p^-$ for some $1\leq p<\infty$.

If $u\in W^{1,p}_a(\omega)$ then, for any $0<\alpha<1$,
$(\mathfrak{D}_{{\rm left},a})^\alpha u$ is a function in $L^p(\E^{-p}\omega)$ such that
$$(\mathfrak{D}_{{\rm left},a})^\alpha u(x)=c_\alpha \E(x)\int_{-\infty}^x\frac{(\E^{-1}u)(x)-(\E^{-1}u)(t)}{(x-t)^{1+\alpha}}\,dt$$
for a.e.~$x\in\R$, and there exists a constant $C=C(\omega)>0$ such that
$$\|(\mathfrak{D}_{{\rm left},a})^\alpha u\|_{L^p(\E^{-p}\omega)}\leq C\|u\|_{W^{1,p}_a(\omega)}.$$
Moreover,
\begin{equation}\label{eq:BBM}
\lim_{\alpha\to1^-}(\mathfrak{D}_{{\rm left},a})^\alpha u=\mathfrak{D}_{{\rm left},a} u\qquad\hbox{in}~L^p(\E^{-p}\omega)~\hbox{and a.e.~in}~\R
\end{equation}
and
$$\lim_{\alpha\to0^+}(\mathfrak{D}_{{\rm left},a})^\alpha u=u\qquad\hbox{a.e.~in}~\R.$$

Conversely, if $u\in L^p(\E^{-p}\omega)$ is such that $(\mathfrak{D}_{{\rm left},a})^\alpha u\in L^p(\E^{-p}\omega)$
and $(\mathfrak{D}_{{\rm left},a})^\alpha u$ converges in $L^p(\E^{-p}\omega)$ as $\alpha\to1^-$, then
$u\in W^{1,p}(\E^{-p}\omega)$ and \eqref{eq:BBM} holds.
\end{theorem}

\begin{proof}
If $u$ is as in the statement, then $\E^{-1}u\in L^p(\omega)$. Furthermore, it is easily verified that $\mathfrak{D}_{{\rm left},a}u\in L^p(\E^{-p}\omega)$
if and only if $\Dl(\E^{-1}u)\in L^p(\omega)$. Thus, $u\in W^{1,p}_a(\omega)$ if and only if $\E^{-1}u\in W^{1,p}(\omega)$. 
By \cite[Theorem~1.1]{mary}, $(\Dl)^\alpha(\E^{-1}u)\in L^p(\omega)$ with
$$(\Dl)^\alpha(\E^{-1}u)(x)=c_\alpha\int_{-\infty}^x\frac{\E^{-1}u(x)-\E^{-1}u(t)}{(x-t)^{1+\alpha}}\,dt\qquad\hbox{for a.e.}~x\in\R.$$
Also, there exists $C=C(\omega)>0$ such that $\|(\Dl)^\alpha(\E^{-1}u)\|_{L^p(\omega)}^p\leq C\|\E^{-1}u\|_{W^{1,p}(\omega)}^p$.
Moreover,
\begin{equation}\label{eq:Dlconvergencemary}
\lim_{\alpha\to1^-}(\Dl)^\alpha(\E^{-1}u)=(\E^{-1}u)'\qquad\hbox{in}~L^p(\omega)~\hbox{and a.e.~in}~\R
\end{equation}
and  $\lim_{\alpha\to0^+}(\Dl)^\alpha(\E^{-1}u)=\E^{-1}u$ a.e.~in $\R$.
We get the parallel conclusions for $(\mathfrak{D}_{{\rm left},a})^\alpha$
in our statement by noticing that $(\mathfrak{D}_{{\rm left},a})^\alpha u=\E(\Dl)^\alpha(\E^{-1}u)$ in the sense of distributions
in $(\mathcal{S}_{-,\E})'$, and, thus, almost everywhere.

Let $u\in L^p(\E^{-p}\omega)$ be such that $(\mathfrak{D}_{{\rm left},a})^\alpha u\in L^p(\E^{-p}\omega)$
and $(\mathfrak{D}_{{\rm left},a})^\alpha u$ converges in $L^p(\E^{-p}\omega)$ as $\alpha\to1^-$. Then
$\E^{-1}u\in L^p(\omega)$, $(\Dl)^\alpha(\E^{-1}u)\in L^p(\omega)$
and $(\Dl)^\alpha(\E^{-1}u)$ converges in $L^p(\omega)$ as $\alpha\to1^-$. By \cite[Theorem~1.1]{mary},
$\E^{-1}u\in W^{1,p}(\omega)$ and \eqref{eq:Dlconvergencemary} holds. We conclude
by using the distributional identity $(\mathfrak{D}_{{\rm left},a})^\alpha u=\E(\Dl)^\alpha(\E^{-1}u)$.
\end{proof}

%%%%%%%%%%%%%%%%%%%%%%%%%%%%%%%%%%%%%%%%%%%%%%%%%%%%%%
\subsection{Vector-valued extensions}
%%%%%%%%%%%%%%%%%%%%%%%%%%%%%%%%%%%%%%%%%%%%%%%%%%%%%%

As a simple consequence of Theorem~\ref{thm:BBM} and a suitable version of Rubio de Francia's extrapolation theorem for one-sided weights, we obtain vector-valued inequalities for $(\mathfrak{D}_{{\rm left},a})^\alpha$. First, we state the extrapolation theorem which can be found
in \cite{MarOrtTor} (see also \cite[Theorem~3.25]{CruMarPer}).

\begin{theorem}\label{thm:extrapolation}
Suppose that for some $1\leq p_0<\infty$, every $\omega_0\in A_{p_0}^-$
and every pair of nonnegative functions $(u,v)$ in a family $\mathcal{F}$,
$$
\int_\R u(x)^{p_0}\omega_0(x)\,dx\leq C(\omega_0) \int_\R v(x)^{p_0} \omega_0(x)\,dx
$$
with $C(\omega_0)>0$ depends only on $\omega_0$. Then, for all $1<p<\infty$ and for all $\omega\in A_p^-$,
$$
\int_\R u(x)^{p}\omega(x)\,dx\leq C(p, p_0, \omega) \int_\R v(x)^{p} \omega(x)\,dx$$
for all $(u,v)\in \mathcal{F}$, where $C(p, p_0, \omega)>0$ depends only on $p$, $p_0$ and $\omega$.
\end{theorem}

\begin{theorem}\label{thm:vector valued}
For all $1<p, r<\infty$, $\omega\in A_p^-$, and $u_1, u_2, \dots, u_n \in W^{1,p}_a(\omega)$, 
\begin{equation}\label{eq:vector valued}
\left\| \left( \sum_{i=1}^n |(\mathfrak{D}_{{\rm left},a})^\alpha u_i|^r \right)^{1/r}\right\|_{L^p(\E^{-p}\omega)}\leq C \left\| \left( \sum_{i=1}^n |u_i|^r + |\mathfrak{D}_{{\rm left},a} u_i|^r\right)^{1/r} \right\|_{L^p(\E^{-p}\omega)}
\end{equation} 
where $C>0$ depends only on $p$, $r$, $\alpha$ and $\omega$.
\end{theorem}

\begin{proof}
Consider the family $\mathcal{F}$ of pairs $(u,v)$ of the form
$$
u=\E^{-1}\left( \sum_{i=1}^n |(\mathfrak{D}_{{\rm left},a})^\alpha u_i|^r \right)^{1/r} \quad \text{and} \quad v=\E^{-1} \left( \sum_{i=1}^n |u_i|^r + |\mathfrak{D}_{{\rm left},a} u_i|^r\right)^{1/r}
$$
for any $n\in \mathbb{N}$ and $u_1, u_2, \dots, u_n \in C_c^\infty(\R)$. Notice that given any $\omega_0 \in A_r^-$ we have that $u_1, u_2, \dots, u_n \in W^{1,r}_a(\omega_0)$. Then, by Theorem~\ref{thm:BBM},
\begin{align*}
\int_{\R} u(x)^r \omega_0(x)\,dx &=\sum_{i=1}^n \int_{\R} \E^{-r}(x) |(\mathfrak{D}_{{\rm left},a})^\alpha u_i(x)|^r \omega_0(x)\,dx\\
&= \sum_{i=1}^n \|(\mathfrak{D}_{{\rm left},a})^\alpha u_i\|^r_{L^r(\E^{-r}\omega_0)}\\
&\leq C \sum_{i=1}^n\Big( \|u_i\|^r_{L^r(\E^{-r}\omega_0)} + \|\mathfrak{D}_{{\rm left},a} u_i\|^r_{L^r(\E^{-r}\omega_0)}\Big)\\
&= C  \int_\R \E^{-r}(x) \sum_{i=1}^n\big(|u_i(x)|^r+  |\mathfrak{D}_{{\rm left},a}u_i(x)|^r\big) \omega_0(x)\,dx\\
&= C \int_\R v(x)^r \omega_0(x)\,dx
\end{align*}
for some constant $C>0$ depending only on $\omega_0\in A_r^-$ (there is also a dependence on $\alpha$, which is implicit when we apply Theorem~\ref{thm:BBM}). Hence, from Theorem~\ref{thm:extrapolation} we deduce \eqref{eq:vector valued} for all $1<p<\infty$, $\omega \in A_p^-$ and $u_1, u_2, \dots, u_n \in C_c^\infty(\R)$. A density argument gives the desired inequality for $u_1, u_2, \dots, u_n \in W^{1,p}_a(\omega)$. 	
\end{proof}

%%%%%%%%%%%%%%%%%%%%%%%%%%%%%%%%%%%%%%%%%%%%%%%%%%%%%%
\subsection{Fundamental Theorem of Fractional Calculus}\label{subsect:TFC}
%%%%%%%%%%%%%%%%%%%%%%%%%%%%%%%%%%%%%%%%%%%%%%%%%%%%%%

We now turn our attention to the fundamental theorem of fractional calculus, or, equivalently, the
solvability of the nonlocal Poisson equation
$$(\mathfrak{D}_{{\rm left},a})^\alpha u=f.$$
For this, we define the negative fractional powers of $\mathfrak{D}_{{\rm left},a}$ (and $\mathfrak{D}_{{\rm right},a}$).

\begin{definition}
Given $\alpha>0$ and functions $f$ and $g$, we define the negative fractional power operators
$$(\mathfrak{D}_{{\rm left},a})^{-\alpha}f(x)
=\frac1{\Gamma(\alpha)}\int_0^\infty e^{-t \mathfrak{D}_{{\rm left},a}}f(x)\,\frac{dt}{t^{1-\alpha}}$$
and
$$(\mathfrak{D}_{{\rm right},a})^{-\alpha}g(x)
=\frac1{\Gamma(\alpha)}\int_0^\infty e^{-t \mathfrak{D}_{{\rm right},a}}g(x)\,\frac{dt}{t^{1-\alpha}},$$
whenever the integrals converge.
\end{definition}

We recall that in \cite{BerMarStiTor} the negative powers of the left and right derivative operators in \eqref{izda}
were defined as
\begin{align*}
(\Dl)^{-\alpha}f(x) &= \frac1{\Gamma(\alpha)}\int_0^\infty e^{-t D_{\rm left}}f(x)\,\frac{dt}{t^{1-\alpha}}\\
&=c_{-\alpha}\int_{-\infty}^x\frac{f(t)}{(x-t)^{1-\alpha}}\,dt
\end{align*}
and 
\begin{align*}
(\Dr)^{-\alpha}g(x) &= \frac1{\Gamma(\alpha)}\int_0^\infty e^{-t D_{\rm right}}g(x)\,\frac{dt}{t^{1-\alpha}}\\
&=c_{-\alpha}\int_x^{\infty}\frac{g(t)}{(t-x)^{1-\alpha}}\,dt,
\end{align*}
where $c_{-\alpha}=1/\Gamma(\alpha)>0$. Then $(\Dl)^{-\alpha}$ is the left-sided Weyl fractional integral of order $\alpha>0$
and $(\Dr)^{-\alpha}$ is the corresponding right-sided version.

The next relation follows from Theorem \ref{dualidadL}. 

\begin{lemma}\label{lem:fractionalconjugation-}
Let $f$ and $g$ be functions such that $\E^{-1}f,\E g\in\mathcal{S}(\R)$. Then, for all $\alpha>0$,
\begin{align*}
(\mathfrak{D}_{{\rm left},a})^{-\alpha}f(x) &= \E(x)(\Dl)^{-\alpha}(\E^{-1}f)(x) \\
&=c_{-\alpha}\E(x)\int_{-\infty}^x\frac{(\E^{-1}f)(t)}{(x-t)^{1-\alpha}}\,dt
\end{align*}
and
\begin{align*}
(\mathfrak{D}_{{\rm right},a})^{-\alpha}g(x) &= \E(x)^{-1}(\Dr)^{-\alpha}(\E g)(x) \\
&=c_{-\alpha}\E^{-1}(x)\int_x^{\infty} \frac{(\E g)(t)}{(t-x)^{1-\alpha}}\,dt.
\end{align*}
In addition,
$$\int_\R[(\mathfrak{D}_{{\rm left},a})^{-\alpha}f(x)]g(x)\,dx=\int_\R f(x)[(\mathfrak{D}_{{\rm right},a})^{-\alpha}g(x)]\,dx.$$
\end{lemma}

Clearly, if $f$ is smooth enough and has sufficient integrability with respect to $\E$ then, by using
Lemmas \ref{lem:fractionalconjugation} and \ref{lem:fractionalconjugation-},
one has $(\mathfrak{D}_{{\rm left},a})^{\alpha}[(\mathfrak{D}_{{\rm left},a})^{-\alpha}f]=f$.
The question of whether this identity holds for $L^p$ functions is much more delicate. In fact, we need to introduce the
truncated fractional operator
$$(\mathfrak{D}_{{\rm left},a})_\varepsilon^{\alpha}u(x)=
c_\alpha \E(x)\int_{-\infty}^{x-\varepsilon}\frac{(\E^{-1}u)(x)-(\E^{-1}u)(t)}{(x-t)^{1+\alpha}}\,dt\qquad\hbox{for}~\varepsilon>0.$$

\begin{theorem}[Fundamental theorem of fractional calculus]\label{TFC}
Let $\E$ be as in \eqref{definicion}. Fix $0<\alpha<1$, $1<p<1/\alpha$, $1/q=1/p-\alpha$, and $\omega \in A_{p,q}^-$.
Then, for any $f \in L^p(\E^{-p}\omega^p)$, we have
$$f(x)=\lim_{\varepsilon\to0^+}(\mathfrak{D}_{{\rm left},a})_{\varepsilon}^\alpha(\mathfrak{D}_{{\rm left},a})^{-\alpha}f(x)$$
for a.e.~$x\in\R$ and in $L^p(\E^{-p}\omega^p)$.
\end{theorem}

\begin{proof}
Observe that $f\in L^p(\E^{-p}\omega^p)$ if and only if $\E^{-1}f\in L^p(\omega^p)$, and so,
by \cite[Theorem~6.3]{BerMarStiTor},
$\lim_{\varepsilon\to0^+}(\Dl)_\varepsilon^\alpha[(\Dl)^{-\alpha}(\E^{-1}f)]=f$ a.e.~and in $L^p(\omega^p)$.
We conclude by noticing that
$$(\Dl)_\varepsilon^\alpha[(\Dr)^{-\alpha}(\E^{-1}f)]=\E^{-1}(\mathfrak{D}_{{\rm left},a})^\alpha_\varepsilon
[(\mathfrak{D}_{{\rm left},a})^{-\alpha}f].$$
\end{proof}

%%%%%%%%%%%%%%%%%%%%%%%%%%%%%%%%%%%%%%%%%%%%%%%%%%%%%%
\subsection{Fractional Sobolev spaces}
%%%%%%%%%%%%%%%%%%%%%%%%%%%%%%%%%%%%%%%%%%%%%%%%%%%%%%

We characterize the fractional Sobolev spaces associated with $(\mathfrak{D}_{{\rm left},a})^\alpha$.
These spaces are defined as the image of $L^p$ under the fractional integral operator $(\mathfrak{D}_{{\rm left},a})^{-\alpha}$.
The last result in this Section can be obtained by using \cite[Theorem~6.5]{BerMarStiTor} and Lemma \ref{lem:fractionalconjugation},
and the details are left to the interested reader.

\begin{theorem}[Characterization of fractional Sobolev spaces]\label{thm:fractionalSobolev}
Let $\E$ be as in \eqref{definicion}. Fix $0<\alpha<1$, $1<p<1/\alpha$, $1/q=1/p-\alpha$, and $\omega \in A_{p,q}^-$.
Let $u$ be a measurable function. The following statements are equivalent.
\begin{enumerate}[$(i)$]
\item The exists $f\in L^p(\E^{-p}\omega^p)$ such that $u=(\mathfrak{D}_{{\rm left},a})^{-\alpha}f$.
\item $u\in L^q(\E^{-q}\omega^q)$ and the limit $\lim_{\varepsilon\to0^+}(\mathfrak{D}_{{\rm left},a})_{\varepsilon}^\alpha u$
exists in $L^p(\E^{-p}\omega^p)$.
\item $u\in L^q(\E^{-q}\omega^q)$ and
$\sup_{\varepsilon>0}\|(\mathfrak{D}_{{\rm left},a})_{\varepsilon}^\alpha u\|_{L^p(\E^{-p}\omega^p)}<\infty$.
\end{enumerate}
\end{theorem}

%%%%%%%%%%%%%%%%%%%%%%%%%%%%%%%%%%%%%%%%%%%%%%%%%%%%%%
\section{New inverse measures and polynomials, and boundedness of singular integrals
in inverse measure spaces}\label{part:orthogonal}
%%%%%%%%%%%%%%%%%%%%%%%%%%%%%%%%%%%%%%%%%%%%%%%%%%%%%%

%%%%%%%%%%%%%%%%%%%%%%%%%%%%%%%%%%%%%%%%%%%%%%%%%%%%%%
\subsection{Hermite polynomials and Hermite functions}
%%%%%%%%%%%%%%%%%%%%%%%%%%%%%%%%%%%%%%%%%%%%%%%%%%%%%%

%%%%%%%%%%%%%%%%%%%%%%%%%%%%%%%%%%%%%%%%%%%%%%%%%%%%%%
\subsubsection{Ornstein--Uhlenbeck operator and Hermite polynomials} 
%%%%%%%%%%%%%%%%%%%%%%%%%%%%%%%%%%%%%%%%%%%%%%%%%%%%%%
 
The Ornstein--Uhlenbeck operator
$$\mathcal{O}=-\frac{d^2}{dx^2}+2x\frac{d}{dx}$$
is self-adjoint with respect to the Gaussian measure $d\gamma (x) = e^{-x^2}dx$ on $\mathbb{R}$.
In order to develop a harmonic analysis associated to $\mathcal{O}$, it is usual to factor
the operator as a composition of first order differential operators:
$$\mathcal{O} =\Big(-\frac{d}{dx}+2x\Big)\frac{d}{dx}$$
see \cite{Muckenhoupt}. The first order differential operator
$-\frac{d}{dx}+2x$ fits into the analysis of Section \ref{part:fractional} and, in particular,
satisfies the hypotheses of Theorems \ref{dualidadL} and \ref{main}.
In fact, $-\frac{d}{dx}+2x=\mathfrak{D}_{{\rm right},a}$, where $a(x) = 2x$ and, by choosing $x_0=0$ in \eqref{definicion}, 
$\mathcal{E}(x) = e^{-x^2}$. Then,
$$\mathfrak{D}_{{\rm right},a}u(x) = -e^{x^2}\frac{d}{dx}\big(e^{-x^2}u(x)\big).$$
It is also well known that a family of orthogonal eigenfunctions of the operator $\mathcal{O}$ is given by the family of Hermite  polynomials $\{H_n(x)\}_{n\geq0}$ and, in fact, $\mathcal{O} H_n =2nH_n$. The family of Hermite polynomials  can be defined
by the  Rodrigues formula (see \cite{Sz})
$$H_n(x) = (-1)^ne^{x^2}\frac{d^ne^{-x^2}}{dx^n}.$$
Observe that, having in mind Lemma~\ref{lem:conjugationalpha1}, the Hermite polynomials can also be described by
$$H_n(x)= \mathfrak{D}^n_{{\rm right},2x}1(x).$$ 

%%%%%%%%%%%%%%%%%%%%%%%%%%%%%%%%%%%%%%%%%%%%%%%%%%%%%%
\subsubsection{Hermite functions, Lebesgue measure and transference}\label{transferenciaHermite-Lebesgue}
%%%%%%%%%%%%%%%%%%%%%%%%%%%%%%%%%%%%%%%%%%%%%%%%%%%%%%

The Hermite operator
$$\mathcal{H} =-\frac{d^2}{dx^2}+x^2$$
is self-adjoint with respect to the Lebesgue measure on $\mathbb{R}$.
Moreover, $\mathcal{H}h_n= (2n+1)h_n$ where $\{h_n(x)\}_{n\geq0}$ denote the 
the family of Hermite functions defined as  
$h_n(x) = H_n(x)e^{-x^2/2}$, where $\{H_n\}_{n\geq0}$ is the family of Hermite polynomials defined  above. 
In order to develop a harmonic analysis associated to $\mathcal{H}$,
the following factorization was introduced (see \cite{Stempak,Thangavelu})
$$\mathcal{H} = \frac12\bigg[\Big(\frac{d}{dx} +x\Big)\Big(-\frac{d}{dx} +x\Big)+\Big(-\frac{d}{dx} +x\Big)
\Big(\frac{d}{dx} +x\Big)\bigg].$$ 
Observe that $\frac{d}{dx} +x=\mathfrak{D}_{{\rm left},a}$  and $-\frac{d}{dx} +x=\mathfrak{D}_{{\rm right},a}$,
where $a(x) = x$.  Hence, by choosing $x_0=0$ in \eqref{definicion}, we have that $\mathcal{E}(x)=e^{x^2/2}$ and, by Lemma~\ref{lem:fractionalconjugation}, for $0<\alpha<1$,
\begin{equation}\label{apareceinversa}
\begin{aligned}
\Big(\frac{d}{dx} +x\Big)^\alpha u(x)&= e^{-x^2/2}\Big(\frac{d}{dx}\Big)^\alpha\big[e^{(\cdot)^2/2}u(\cdot)\big](x) \\
\Big(-\frac{d}{dx} +x\Big)^\alpha v(x)&= e^{x^2/2}\Big(-\frac{d}{dx}\Big)^\alpha\big[e^{-(\cdot)^2/2}v(\cdot)\big](x). 
\end{aligned}
\end{equation}

The weight $\E(x)=e^{-x^2/2}$, associated to the operator $-\frac{d}{dx}+x$, was  used in \cite{AbuStinTor} and \cite{AbuTor}
to establish the following connection between the Hermite and Ornstein--Ulhenbeck operators.  
Let $U$ be the isometry  from  $L^2(\mathbb{R}, e^{-x^2} dx )$ into $L^2(\mathbb{R},dx)$ defined by
$$Uf(x) =e^{-x^2/2}f(x).$$
It can be checked that (see \cite{AbuTor})
\begin{align*}
U\Big[\Big(-\frac{d}{dx}+2x\Big)f\Big]& = \Big(-\frac{d}{dx}+x\Big)Uf \\
U\Big[\frac{d}{dx}f\Big]&= \Big(\frac{d}{dx}+x\Big)Uf.
\end{align*}
Hence, for smooth functions $f \in L^2(\mathbb{R}, e^{-x^2} dx )$,
\begin{equation}\label{gaussian1}
\begin{aligned} 
U(\mathcal{O} f )&= U\bigg[\Big(-\frac{d}{dx}+2x\Big)\frac{d}{dx}f\bigg] \\
&=\Big(-\frac{d}{dx}+x\Big)\Big(\frac{d}{dx}+x\Big)Uf \\
&=\frac12\Big(\Big(\frac{d}{dx}+x\Big)\Big(-\frac{d}{dx}+x\Big)+\Big(-\frac{d}{dx}+x\Big)\Big(\frac{d}{dx}+x\Big)\Big)Uf- Uf \\
&=\mathcal{H}(Uf)-Uf.
\end{aligned}
\end{equation}

%%%%%%%%%%%%%%%%%%%%%%%%%%%%%%%%%%%%%%%%%%%%%%%%%%%%%%
\subsubsection{Inverse Gaussian measure and new polynomials of Hermite type.
Proofs of Theorem \ref{inverHer} and Theorem \ref{inverAll}$(i)$}\label{transferenciaHermite-inverse}
%%%%%%%%%%%%%%%%%%%%%%%%%%%%%%%%%%%%%%%%%%%%%%%%%%%%%%

The formulas appearing in \eqref{apareceinversa} show that the function
$e^{-x^2/2}$ arises naturally in relation with the operator $-\frac{d}{dx}+x$.
Moreover, the measure $e^{-x^2}dx$ is related with the Ornstein--Uhlenbeck operator. 
However, the formulas appearing in \eqref{apareceinversa} equally show
that the function $e^{x^2/2}$ appears naturally in relation with the operator $\frac{d}{dx}+x$.
This raises another differential operator, parallel to the Ornstein--Uhlenbeck case,
which will be self-adjoint with respect to the measure $e^{x^2}dx$.
In order to find that operator we shall use the following isometry, which is analogous to the isometry
$U$ defined in the previous paragraph. 

Define the isometry $\widetilde{U}$ from $L^2(\mathbb{R}, e^{x^2}dx)$ into $L^2(\mathbb{R},dx)$as
$$\widetilde{U}(x)= e^{x^2/2}f(x).$$
The following identities can be checked:
\begin{align*}
\widetilde{U}\Big[\Big(\frac{d}{dx}+2x\Big)f\Big] &= \Big(\frac{d}{dx}+x\Big)\widetilde{U}f \\
\widetilde{U}\Big[\frac{d}{dx}f\Big] &= \Big(\frac{d}{dx}-x\Big)\widetilde{U}f.
\end{align*}
Hence, if we define the second order operator  
$$\widetilde{\mathcal{O}} = \frac{d^2}{dx^2}+2x\frac{d}{dx},$$
for smooth functions we get
\begin{equation}\label{gaussian2}
\begin{aligned} 
\widetilde{U}(\widetilde{\mathcal{O}} f) &= \widetilde{U}\Big[\Big(\frac{d^2}{dx^2}+2x\frac{d}{dx}\Big)f\Big] \\
&= \Big(\frac{d}{dx}+x\Big)\Big(\frac{d}{dx}-x\Big)\widetilde{U}f \\
&=- \mathcal{H}(\widetilde{U} f) - \widetilde{U}f.
\end{aligned}
\end{equation}
By using \eqref{gaussian1} and \eqref{gaussian2},
\begin{align}\label{relationtilde}  \nonumber 
\widetilde{\mathcal{O}} f&=-\widetilde{U}^{-1}\mathcal{H}(\widetilde{U}f)-f \\
&= -\widetilde{U}^{-1}U \mathcal{O} U^{-1}\widetilde{U}f -2f \\
&= -e^{-x^2} \mathcal{O}(e^{(\cdot)^2} f(\cdot))-2f \nonumber
\end{align}
Hence, as the family $\{H_n\}_{n\geq0}$ of Hermite polynomials are eigenfunctions of the operator
$\mathcal{O}$ with eigenvalues $2n$, the family of functions $H^*_n(x)= \gamma(x)H_n(x)$, $n\geq0$, $\gamma(x)=e^{-x^2}$,
is a family  of eigenfunctions (which \emph{are not polynomials}) of the operator $\widetilde{\mathcal{O}}$ with eigenvalues $-(2n+2)$.
In this case, the family $\{H^*_n\}_{n\geq0}$ is an orthonormal family with respect to the inverse Gaussian measure
$d\gamma_{-1}(x) =e^{x^2}dx$. These ideas were also considered in \cite{jorge,italiano,expanders}.
Here, we can say something more.
Consider the isometry
\begin{eqnarray*} E: L^2(\mathbb{R}, \gamma(x) dx ) &\longrightarrow& L^2(\mathbb{R}, \gamma_{-1}(x) dx )\\
  f &\longrightarrow&  f (\cdot)  \gamma(\cdot).
\end{eqnarray*} 
Observe that $E( H_n)(x) = H^*_n(x)$, hence as $\{H_n\}_{n\geq0}$ is an orthonormal basis of
$ L^2(\mathbb{R}, \gamma(x) dx ) $, it follows that $H^*_n$ is an orthonormal basis of
$L^2(\mathbb{R}, \gamma_{-1}(x) dx )$.

Next we find, via a reproducing formula argument,
a family of \emph{polynomials} that are eigenfunctions of the operator $\widetilde{\mathcal{O}}$.
Consider the function
$$w(x,t)= e^{t^2-2xt},\qquad\hbox{for}~x,t\in\R,$$
and its  Taylor expansion around $t=0$, for $x$ fixed:
$$w(x,t) = \sum_{n=0}^\infty\, \frac1{n!} \frac{\partial^n w}{\partial t^n}\Big |_{t=0} \, t^n.$$
We observe that  
$$\frac{\partial^n w}{\partial t^n}\Big |_{t=0}=e^{-x^2} \frac{\partial^n e^{(x-t)^2}}{\partial t^n}\Big |_{t=0}
=(-1)^ne^{-x^2}\frac{\partial^n e^{r^2}}{\partial r^n}\Big |_{r=x} =: \widetilde{H} _n(x).$$
In other words, the newly defined polynomials $\widetilde{H} _n(x)$ are given by a Rodrigues formula.
The function $w$ satisfies the first order ODEs
$$\frac{\partial w}{\partial t} +(2x-2t) w=0$$
and 
\begin{equation}\label{genHer}
\frac{\partial w}{\partial x}+2t w =0.
\end{equation}
Therefore, from the first ODE,
$$\sum_{n=0}^{\infty} \frac{\widetilde{H} _{n+1}(x)}{n!} t^n \, + 2x \sum_{n=0}^{\infty} \frac{\widetilde{H} _{n}(x)}{n!} t^n -2  \sum_{n=0}^{\infty} \frac{\widetilde{H} _{n}(x)}{n!} t^{n+1}=0.$$
Consequently,
$$\widetilde{H} _{n+1}(x)+2x \widetilde{H} _n(x)-2n\widetilde{H} _{n-1}(x) =0.$$
On the other hand, by using \eqref{genHer}, we have
$$\sum_{n=0}^{\infty} \frac{\widetilde{H}'_{n}(x)}{n!} t^n+2 \sum_{n=0}^{\infty} \frac{\widetilde{H} _{n}(x)}{n!} t^{n+1}=0,$$
which implies
\begin{equation}\label{derivacion}
\widetilde{H} _n' (x)=-2n\widetilde{H} _{n-1}(x)
\end{equation}
and, thus,
$$\widetilde{H} _{n+1}(x)+2x \widetilde{H} _n(x)+\widetilde{H}'_n(x) =0.$$
We differentiate the expression above and use \eqref{derivacion} to get
\begin{align*}
0&=\widetilde{H}'_{n+1} +2 \widetilde{H} _n+2x\widetilde{H}'_n+\widetilde{H}''_n\\
&= -2(n+1)\widetilde{H} _n+2\widetilde{H} _n+2x\widetilde{H}_n'+\widetilde{H}''_n \\
&= \widetilde{H}''_n+2x\widetilde{H}'_n-2n\widetilde{H} _n.
\end{align*}
Therefore, the family of polynomials $\{\widetilde{H} _n\}_{n\geq0}$ are eigenfunctions of $\widetilde{O}$, with 
\begin{equation}\label{autovaloresinversos}
\widetilde{O}\widetilde{H}_n =2n\widetilde{H} _n. 
\end{equation} 

Observe that the operator $\widetilde{O}$ is self-adjoint with respect to the inverse Gaussian measure $e^{x^2}dx$.
Furthermore, the family of polynomials $\{\widetilde{H} _n\}_{n\geq0}$ that we have found here,
even though they satisfy \eqref{autovaloresinversos}, do not belong to the space $L^2(\mathbb{R},e^{x^2}dx)$. 

Observe that  $\widetilde{H}_n(x)=(-1)^n \mathfrak{D}^n_{{\rm left},2x}1(x)=\mathfrak{D}^n_{{\rm right},-2x}1(x)$.

\begin{rem}[Hermite functions of degree $\alpha\in\R$]
Taking into account the description $$H_n(x)= \mathfrak{D}^n_{{\rm right},2x}1(x)$$  of the Hermite polynomials and the fractional power
operators $\mathfrak{D}_{{\rm right},2x}^{\alpha}$ and $\mathfrak{D}_{{\rm right},2x}^{-\alpha}$ defined in
Sections~\ref{subsect:fractionalpositive} and \ref{subsect:TFC}, it is natural to consider
\emph{Hermite functions of degree $\alpha$} defined by
$$H_\alpha(x)= \mathfrak{D}^\alpha_{{\rm right},2x}1(x)= \frac1{\Gamma(-\alpha)} \int_0^\infty\big(e^{-2xt-t^2}-1\big)\,\frac{dt}{t^{1+\alpha}} \quad \text{for $0<\alpha<1$},$$
and
\begin{equation}\label{H alpha neg}
H_{-\alpha}(x)= \mathfrak{D}^{-\alpha}_{{\rm right},2x}1(x)=\frac1{\Gamma(\alpha)}\int_0^\infty e^{-2xt-t^2}\,\frac{dt}{t^{1-\alpha}} \quad \text{for $\alpha>0$.}
\end{equation}
In \cite[Chapter~10.2]{Leb}, Hermite functions of degree $\nu\in\mathbb{C}$ are defined as particular solutions of the ODE
$$
y''-2xy'+2\nu y=0,
$$
that is, the Hermite functions $H_\nu(x)$ are eigenfunctions of the operator
$\mathcal{O}$ with eigenvalues $2\nu$. Moreover, when $\mathrm{Re}(\nu)<0$, the following integral formula holds
$$
H_{\nu}(x)= \frac1{\Gamma(-\nu)}\int_0^\infty e^{-2xt-t^2}\,\frac{dt}{t^{1+\nu}}
$$
(see Section~10.5 in \cite{Leb}). Putting $\nu=-\alpha$ in \eqref{H alpha neg}, we see that $H_{-\alpha}(x)= \mathfrak{D}^{-\alpha}_{{\rm right},2x}1(x)$ coincides with the Hermite functions defined in \cite{Leb}.
However, proving that \emph{our} Hermite functions of degree $0<\alpha<1$ coincide with the ones considered in \cite{Leb}
would amount to consider a Taylor expansion for the generating function involving
fractional derivatives. We leave this open problem to the interested reader.
\end{rem}

%%%%%%%%%%%%%%%%%%%%%%%%%%%%%%%%%%%%%%%%%%%%%%%%%%%%%%
\subsection{Laguerre polynomials and Laguerre functions}
%%%%%%%%%%%%%%%%%%%%%%%%%%%%%%%%%%%%%%%%%%%%%%%%%%%%%%

Along this subsection, unless otherwise stated, we assume that $\alpha>-1$ is fixed.

%%%%%%%%%%%%%%%%%%%%%%%%%%%%%%%%%%%%%%%%%%%%%%%%%%%%%%
\subsubsection{Laguerre operator and Laguerre polynomials}
%%%%%%%%%%%%%%%%%%%%%%%%%%%%%%%%%%%%%%%%%%%%%%%%%%%%%%

The Laguerre operator 
$$\mathfrak{L}_\alpha =x \frac{d^2}{dx^2} +(\alpha+1-x ) \frac{d}{dx},\qquad\hbox{for}~x\in(0,\infty),$$
is self-adjoint with respect to the Laguerre measure $d\mu_\alpha(x)= e^{-x} x^{\alpha} dx$ in $(0,\infty)$.
The family of Laguerre polynomials $\{L_n^\alpha(x)\}_{n\geq0}$ of type $\alpha$ is defined via the Rodrigues formula
$$L^\alpha_n(x) = e^x \frac{x^{-\alpha}}{n!} \frac{d^n}{d x^n} (e^{-x} x^{n+\alpha}).$$
Moreover, the Laguerre polynomials are orthogonal with respect to the Laguerre measure and 
eigenfunctions of the Laguerre operator with
$$\mathfrak{L}_\alpha L_n^\alpha(x)=-n L_n^\alpha (x).$$
See \cite[p.~100]{Sz} and \cite[p.~7]{Thangavelu}.

In order to develop a harmonic analysis associated to the Laguerre operator, the following factorization was introduced in \cite{InGuTo}:
$$ \mathfrak{L}_\alpha= \operatorname{div}_\alpha \operatorname{grad}_\alpha,$$
where
$$ \operatorname{div}_\alpha f= \sqrt{x} \Big( \frac{d f}{dx}+\Big(\frac{\alpha+1/2}{x} -1\Big) f \Big)\qquad
 \operatorname{grad}_\alpha f = \sqrt{x} \frac{d f}{dx}$$
satisfy
$$\int_0^\infty \operatorname{div}_\alpha f(x)  g(x) \,d \mu_\alpha = -\int_0^\infty f(x) \operatorname{grad}_\alpha g(x) \,d \mu_\alpha$$
for sufficiently smooth functions $f$ and $g$ on $(0,\infty).$
We observe that 
$$\operatorname{div}_\alpha f(x)= 
 \mathfrak{D}_{{\rm left},(\frac{\alpha}{x}-1)} (\sqrt{\cdot}f(\cdot))(x).$$
To apply the ideas developed Section \ref{part:fractional}, we choose $a(x)=\alpha/x-1$, for $x\in (0,\infty)$,   and the point $x_0=1$ in
the definition of the function $\mathcal{E}$ in \eqref{definicion} to get
$$\mathcal{E}(x) = \exp\bigg[-\int_1^x\Big(\frac{\alpha}{y}-1\Big)\,dy\bigg]= x^{-\alpha} e^{x-1}=\frac1{e} x^{-\alpha} e^{x}.$$
Hence, by Lemma~\ref{lem:conjugationalpha1},
$$\operatorname{div}_\alpha f(x)=
x^{-\alpha} e^{x}D_{\rm left}\big(\sqrt{\cdot} f(\cdot) (\cdot)^{\alpha} e^{-(\cdot)}\big)(x)= x^{-\alpha} e^{x}\frac{d}{dx}(e^{-x}x^\alpha x^{1/2}f(x)). $$
Therefore, as in the case of Hermite polynomials, we have the following description of Laguerre polynomials:
$$L_n^\alpha(x) =  \mathfrak{D}^n_{{\rm left},(\frac{\alpha}{x}-1)}\bigg(\frac{(\cdot)^n}{n!}\bigg)(x).$$

%%%%%%%%%%%%%%%%%%%%%%%%%%%%%%%%%%%%%%%%%%%%%%%%%%%%%%
\subsubsection{Laguerre functions, Lebesgue measure and transference}\label{funcionLaguerre}
%%%%%%%%%%%%%%%%%%%%%%%%%%%%%%%%%%%%%%%%%%%%%%%%%%%%%%

Consider the operator
$${\bf L}_\a = -x \frac{d^2}{dx^2}-\frac{d}{dx} +
\frac{x}{4}+\frac{\a^2}{4x},\qquad\hbox{for}~x\in(0,\infty).$$
It is well-known,  see\cite{Sz}, that ${\bf L}_\a$ is non-negative and
self-adjoint with respect to the Lebesgue measure on $(0,\infty)$.
Furthermore, its eigenfunctions are the Laguerre functions defined by
$$\L^\a_n(x) =\Big(\frac{\Gamma(n+1)}{\Gamma(n+\a+1)}\Big)^{1/2}e^{-x/2}x^{\a/2}L^\a_n(x),\qquad n\geq0,~x\in(0,\infty),$$
where $L^\a_n$ are the Laguerre polynomials of type $\a$,
with ${\bf L}_\a \L^\a_n = \big( n+\frac{\alpha+1}{2}\big)   \L^\a_n$.
The orthogonality of the Laguerre polynomials with respect to the Laguerre measure $d\mu_\alpha$
leads to the orthogonality of the family $\{\L^\a_n\}_{n\geq0}$ in $L^2((0,\infty),dx)$.

Parallel as we performed in Subsection \ref{transferenciaHermite-Lebesgue}, we establish a natural
isometry which transfers  the operators $\operatorname{div}_\alpha$ and $\operatorname{grad}_\a$
 defined in the case of Laguerre polynomials into
 first order differential operators $\delta_x^\alpha$
and $(\delta_x^\alpha)^*$ defined for the case of Laguerre functions. Indeed, consider the isometry 
$Q_\a$ from $L^2((0,\infty),d\mu_\alpha)$ into $L^2((0,\infty),dx)$ given by
$$Q_\a f(x) = x^{\alpha/2}e^{-x/2}f(x).$$
Then
 \begin{equation}\label{L3}
 Q_\a(\operatorname{grad}_\alpha f)=\delta^\alpha_x(Q_\a f)
\end{equation}
and 
\begin{equation}\label{L4}
Q_\a(\operatorname{div}_\alpha f)=
-(\delta^\alpha_x)^*(Q_\a f),
\end{equation}
where
$$\delta^\a_x=
\sqrt{x}\frac{d}{dx}+\frac12\Big(\sqrt{x}-\frac{\a}{\sqrt{x}}\Big)
\qquad\hbox{and}\qquad
(\delta^\a_x)^*= -\sqrt{x}\frac{d}{dx}+\frac12\Big(\sqrt{x}-\frac{\a+1}{\sqrt{x}}\Big).$$
These operators were first introduced in \cite{HarTorVivi} in order to build the Riesz transforms for $\mathbf{L}_\a$.
It is shown in \cite{HarTorVivi} that the actions of $\delta^\a_x$ and $(\delta^\a_x)^*$ on Laguerre functions are given by
$$\delta_x^\a(\L^\a_n) = -\sqrt{n}\L^{\a+1}_{n-1},~\hbox{for}~\a >-1,
\qquad\hbox{and}\qquad(\delta^\a_x)^*(\L^{\a+1}_n)=-\sqrt{n+1}\L^{\a-1}_{n+1},~\hbox{for}~\alpha>0.$$
Moreover,
\begin{equation}\label{T1}
\begin{aligned}
(\delta_x^\alpha)^*\delta_x^\alpha g &=
\Big(-\sqrt{x}\frac{d}{dx}+\frac12\Big(\sqrt{x}-\frac{\a+1}{\sqrt{x}}\Big)\Big)\Big(\sqrt{x}\frac{dg}{dx}+\frac12(\sqrt{x}-\frac{\a}{\sqrt{x}}\Big)g\Big)\\
&= -x\frac{d^2 g}{dx^2} - \frac{d g}{dx} +\frac{x}{4} +\frac{\alpha^2}{4x} g -\frac{\alpha+1}{2}g\\
&={\bf L}_\a g-\Big(\frac{\a+1}{2}\Big)g. 
\end{aligned}
\end{equation}
Pasting together identities \eqref{L3}, \eqref{L4} and \eqref{T1}, we get that
\begin{equation}\label{primeratrans}
\mathfrak{L}_\alpha f= \operatorname{div}_\alpha \operatorname{grad}_\a f= -Q_\a^{-1} (\delta_x^\alpha)^* 
\delta_x^\alpha Q_\a f = -Q_\a^{-1} {\bf L}_\alpha Q_\a f +\Big(\frac{\alpha+1}{2}\Big)f.
\end{equation}
Analogously, it can be proved that
\begin{equation}\label{Laguerres}
{\bf L}_\a-\Big(\frac{\a}{2}\Big) = (\delta_x^{\a-1})(\delta_x^{\a-1})^*,\qquad\hbox{for}~\alpha >1.
\end{equation}

Next, observe the following.
\begin{enumerate}[(A)] 
\item It is easy to check that 
$$\delta^\a_xf(x)=\mathfrak{D}_{{\rm left},\frac12(1-\frac{\alpha+1}{x})}(\sqrt{(\cdot)}f(\cdot)) (x).$$
Consequently, by choosing $x_0=1$ in \eqref{definicion},
\begin{align*}
\mathcal{E}_1(y) = \exp\bigg[-\int_1^x \frac{1}{2}(1-\frac{\alpha+1}{y})\,dy\bigg] 
=e^{1/2}x^{(\alpha+1)/2}e^{-x/2}.
\end{align*}
\item Analogously,
$$(\delta^\a_x)^*f(x)=\mathfrak{D}_{{\rm right},\frac12(1-\frac{\alpha}{x})}(\sqrt{(\cdot)}f(\cdot))(x),$$
and
$$\mathcal{E}_2(y) = e^{1/2}x^{\alpha/2} e^{-x/2}.$$
\end{enumerate}

%%%%%%%%%%%%%%%%%%%%%%%%%%%%%%%%%%%%%%%%%%%%%%%%%%%%%%
\subsubsection{Inverse Laguerre measure and new polynomials of Laguerre type.
Proofs of Theorem \ref{inverLag} and Theorem \ref{inverAll}$(ii)$}\label{transferenciaLaguerre}
%%%%%%%%%%%%%%%%%%%%%%%%%%%%%%%%%%%%%%%%%%%%%%%%%%%%%%

We have seen in (A) and (B) in Subsection \ref{funcionLaguerre} that the weights $x^{(\alpha+1)/2}e^{-x/2}$
and $x^{\alpha/2}e^{-x/2}$ appear in a natural way. 
Parallel to the Hermite case, see Subsection \ref{transferenciaHermite-inverse},
we consider a transference isometry from the inverse Laguerre measure space
$L^2((0,\infty), x^{-(\alpha+1)} e^{x}dx)$ into $L^2((0,\infty), dx)$ given by
$$\widetilde{Q}_{\alpha+1}f(x) =  x^{-(\alpha+1)/2}e^{x/2}f(x).$$
Then,
\begin{equation}\label{L1}
\begin{aligned}
(\delta_x^{\alpha})^*&(\widetilde{Q}_{\alpha+1}f)=
 -\sqrt{x} \frac{d(\widetilde{Q}_{\alpha+1}f)}{dx} + \frac12\Big(\sqrt{x}-\frac{\alpha+1}{\sqrt{x}} \Big)\widetilde{Q}_{\alpha+1}f\\
 &=-\sqrt{x}\Big\{\Big(-\frac{\alpha+1}{2x}+\frac12\Big) f +\frac{df}{dx}  \Big\}
x^{-\frac{1+\alpha}{2}}e^{x/2}+ \frac12\Big(\sqrt{x}-\frac{\alpha+1}{\sqrt{x}} \Big)\widetilde{Q}_{\alpha+1}f \\
&=-\widetilde{Q}_{\alpha+1}\Big(\sqrt{x} \frac{d f}{dx}\Big).
\end{aligned}
\end{equation}
On the other hand,
\begin{equation}\label{L2}
\begin{aligned}
(\delta_x^{\alpha})(\widetilde{Q}_{\alpha+1}f)
&=\sqrt{x} \frac{d(\widetilde{Q}_{\alpha+1} f)}{dx} +\frac{1}{2}\Big( \sqrt{x} - \frac{\alpha}{\sqrt{x}}\Big) \widetilde{Q}_{\alpha+1}f \\
&=\sqrt{x}\Big\{\Big(-\frac{\alpha+1}{2 x}+\frac12\Big) f +\frac{df}{dx}  \Big\}
x^{-\frac{1+\alpha}{2}}e^{x/2} +\frac12 (\sqrt{x} - \frac{\alpha}{\sqrt{x}}) \widetilde{Q}_{\alpha+1}f \\
&=  \widetilde{Q}_{\alpha+1} \Big\{\sqrt{x}\Big( \frac{df}{dx}-\Big(\frac{\alpha+1/2}{x} -1\Big) f\Big)\Big\}. 
\end{aligned}
\end{equation}
We can continue now with the identities in \eqref{primeratrans}: by using \eqref{Laguerres}, \eqref{L1} and \eqref{L2}, we get 
\begin{equation*}
\begin{aligned}
\mathfrak{L}_\alpha f &=  -Q^{-1}_\alpha (\delta_x^\alpha)^* \delta_x^\alpha Q_\alpha f \\
&=-Q^{-1}_\alpha  \Big({\bf L}_\alpha - \frac{\alpha}{2} \Big)Q_\alpha f - \frac12f \\
&= -Q^{-1}_\alpha  (\delta_y^{\alpha-1} )(\delta_y^{\alpha-1})^* Q_\alpha f - \frac12 f.
\end{aligned}
\end{equation*}
Notice that 
$$ (\delta_y^{\alpha-1} )(\delta_y^{\alpha-1})^*f=
\widetilde{Q}_\alpha\Big\{\sqrt{x}\Big( \frac{d}{dx}-\Big(\frac{\alpha-1/2}{x} -1\Big)\Big)\Big\}
\Big\{\sqrt{x} \frac{d}{dx} \Big\}\widetilde{Q}_\alpha^{-1}f.$$
In view of these identities, we define the second order differential operator
\begin{align*}
\widetilde{\mathfrak{L}}_\alpha f &:=
\Big(\sqrt{x}\Big( \frac{d}{dx}-\Big(\frac{\alpha-1/2}{x} -1\Big)\Big)\Big(\sqrt{x} \frac{df}{dx} \Big) \\
&=x\frac{d^2f}{dx^2} + (-\alpha +x+1) \frac{df}{dx}.
\end{align*}
From \eqref{L1},  \eqref{L2} and \eqref{Laguerres}, we have 
\begin{align*}
\widetilde{\mathfrak{L}}_\alpha f &= \widetilde{Q}_{\alpha}^{-1} \delta_x^{\alpha -1}(\delta_x^{\alpha-1})^* \widetilde{Q}_{\alpha} f \\
&= \widetilde{Q}_{\alpha}^{-1} (\delta_x^{\alpha})^*\delta_x^{\alpha} \widetilde{Q}_{\alpha} f +\frac12 f 
= \widetilde{Q}_{\alpha}^{-1} Q {\mathfrak{L}}_\alpha Q^{-1}\widetilde{Q}_{\alpha} f +\frac12 f 
\end{align*}
Hence, parallel to the Hermite and Laguerre cases,
we can give  a family of orthonormal eigenfunctions of  $\widetilde{\mathfrak{L}}_\alpha$, that is,
$\{\mathcal{L}^{\alpha,*}_n(x)\}_{n\ge0} = \{\widetilde{Q}_{\alpha}^{-1} Q_\alpha  L_n^\alpha\}_{n\geq0}=\{x^{\alpha} e^{-x}L_n^\alpha(x)\}_{n\geq0}$. Moreover the family $\{\mathcal{L}^{\alpha,*}_n(x)\}$ is a basis in $L^2((0,\infty), e^x x^{-\alpha}dx)$.

Next, via a generating function, we find 
a new family of polynomials $\{\widetilde{L}^\alpha_n\}_{n\geq0}$ that are eigenfunctions of $\widetilde{\mathfrak{L}}_\a$, satisfying
$\widetilde{\mathfrak L}_\alpha  \widetilde{L}^\alpha_n = n \widetilde{L}^\alpha_n$. 
Consider the function
$$w( x,t) = (1-t)^{\alpha-1} e^{xt/(1-t)} = \sum_{n=0}^\infty c_n^\alpha(x) t^n, \qquad\hbox{for}~x\in(0,\infty),~|t|< 1.$$
As $w$ is analytic, the coefficients $c_n^\alpha(x)$ can be written as 
\begin{equation*}
\begin{aligned}
c_n^\alpha(x) &= \frac1{2\pi i} \int_C(1-t)^{\alpha-1}e^{xt/(1-t)} t^{-n-1}\,dt \\
&= \frac{e^{-x}x^{\alpha}}{2\pi i} \int_{C'}\frac{e^{r}r^{n-\alpha}}{(r-x)^{n+1}}\,dr\\
&= \frac{e^{-x}x^{\alpha}}{n!} \frac{d^n}{dx^n}(e^{x} x^{n-\alpha}) =: \widetilde{L}_n^\alpha(x),
\end{aligned}
\end{equation*}
where we made the change of variables $r= \frac{x}{1-t}$. The contour $C$ is chosen
surrounding the point $t=0$ and lying inside the disk $|t| < 1$ in such a way that $C'$
is a small contour surrounding the point $r=x$.
This last identity shows the Rodrigues formula defining our new family $\widetilde{L}^\alpha_n$. 
In order to simplify the notation, we will next write $\widetilde{L}_n$ for $\widetilde{L}^\alpha_n$.

The function $w$ satisfies the ODEs
\begin{align*}
&(1-t)^2\frac{\partial w}{\partial t} +\big( (1-t)(\alpha-1)-x\big) w = 0, \\
&(1-t) \frac{\partial w}{\partial x} -t w=0.
\end{align*}
From these identities we get the following recurrence relation 
\begin{equation}\label{E3}
(n+1)\widetilde{L}_{n+1}+(\alpha-1-x-2n)\widetilde{L}_n +(n-\alpha)\widetilde{L}_{n-1} =0,
\end{equation}
and \begin{equation}\label{E4}
\widetilde{L}_n'- \widetilde{L}_{n-1}'-\widetilde{L}_{n-1}   =0.
\end{equation}
We differentiate \eqref{E3} to get 
\begin{equation}\label{E5}
(n+1)\widetilde{L}'_{n+1}-\widetilde{L}_n+(\alpha-1-x-2n)\widetilde{L}'_n +(n-\alpha)\widetilde{L}'_{n-1} =0.
\end{equation}
We use \eqref{E4} with $n+1$ in place of $n$ and eliminate $\widetilde{L}'_{n+1}$ from \eqref{E5} to obtain
$$(n+1)(\widetilde{L}'_{n}+ \widetilde{L}_n) -\widetilde{L}_n+(\alpha-1-x-2n)\widetilde{L}'_n +(n-\alpha)\widetilde{L}'_{n-1} =0,$$
or, which is the same,
\begin{equation}\label{E7}
-x\widetilde{L}_n' +(\alpha-n)\widetilde{L}_n'+ n\widetilde{L}_n +(n-\alpha) \widetilde{L}_{n-1}'=0.
\end{equation}
Use \eqref{E4} in order to rewrite $\widetilde{L}_{n-1}'$ and find that 
$$-x\widetilde{L}_n' +(\alpha-n)\widetilde{L}_n'+ n\widetilde{L}_n +(n-\alpha) (\widetilde{L}_{n}'-\widetilde{L}_{n-1})=0,$$
that is,
$$-x\widetilde{L}_n' + n\widetilde{L}_n -(n-\alpha) \widetilde{L}_{n-1}=0.$$
Differentiation of this equation gives
$$-x\widetilde{L}_n''- \widetilde{L}_n' + n\widetilde{L}'_n -(n-\alpha) \widetilde{L}'_{n-1}=0.$$
By using \eqref{E7} we get 
$$-x\widetilde{L}_n''- \widetilde{L}_n' + n\widetilde{L}'_n -x\widetilde{L}_n' +(\alpha-n) \widetilde{L}_n' +n \widetilde{L}_n=0,$$
that is,
$$-x\widetilde{L}_n'' + (\alpha+1-x)\widetilde{L}'_n +n \widetilde{L}_n=0,$$
or, in other words,
\begin{equation}\label{autovaloresinversos2}
\widetilde{\mathfrak L}_\a  \widetilde{L}_n = n \widetilde{L}_n.
\end{equation}

Observe that he operator $\widetilde{\mathfrak L}_\a$ is self-adjoint with respect to the inverse Laguerre measure  $e^{x}x^{-\alpha} dx$.
Furthermore, the family of polynomials $\{\widetilde{L} _n\}_{n\geq0}$ that we have found here,
even though they satisfy \eqref{autovaloresinversos2}, do not belong to the space $L^2((0,\infty),e^{x}x^{-\alpha}dx)$. 

Finally,
$$\widetilde{L}_n(x) =  (-1)^n\mathfrak{D}^n_{{\rm right},(\frac{\alpha}{x}-1)}\bigg(\frac{(\cdot)^n}{n!}\bigg)(x)=\mathfrak{D}^n_{{\rm left},(1-\frac{\alpha}{x})}\bigg(\frac{(\cdot)^n}{n!}\bigg)(x).$$

%%%%%%%%%%%%%%%%%%%%%%%%%%%%%%%%%%%%%%%%%%%%%%%%%%%%%%
\subsection{Jacobi polynomials}
%%%%%%%%%%%%%%%%%%%%%%%%%%%%%%%%%%%%%%%%%%%%%%%%%%%%%%

Along this subsection, unless otherwise stated, we assume that $\alpha,\beta>-1$ are fixed.

%%%%%%%%%%%%%%%%%%%%%%%%%%%%%%%%%%%%%%%%%%%%%%%%%%%%%%
\subsubsection{Jacobi operator and Jacobi polynomials}
%%%%%%%%%%%%%%%%%%%%%%%%%%%%%%%%%%%%%%%%%%%%%%%%%%%%%%

Following \cite{Sz}, we define the 
family of Jacobi  polynomials $\{P_n^{(\alpha,\beta)}\}_{n\geq0}$ by  
\begin{align*}
P_n^{(\alpha,\beta)} (x)&=\frac{(-1)^{n}}{2^{n}n!} (1-x)^{-\alpha} (1+x)^{-\beta}\frac{d^n}{dx^n} \Big( (1-x)^{\alpha+n}(1+x)^{\beta+n}\Big)\\
&= \frac{(-1)^{n}}{2^n} (1-x)^{-\alpha} (1+x)^{-\beta}\frac{1}{2\pi i} 
\int_C  \frac{ (1-r)^{\alpha+n}(1+r)^{\beta+n}}{(r-x)^{n+1}}\,dr
\end{align*}
for $x\in[-1,1]$. It is known, see \cite{Sz}, that
\begin{enumerate}[$(1)$]
\item the Jacobi polynomials are orthogonal with respect to the Jacobi measure
$$d\mu_{\alpha,\beta} = (1-x)^\alpha(1+x)^\beta dx\qquad\hbox{on}~(-1,1);$$
\item $P_n^{\alpha,\beta}(x)$ satisfies the differential equation
$$(1-x^2)y''+[\beta-\alpha -(\alpha+\beta+2)x]\,y'+n(n+\alpha+\beta+1)y=0.$$
\end{enumerate}
Therefore, if we define the Jacobi differential operator 
$$\mathcal{G}_{\alpha,\beta} = (1-x^2)\frac{d^2}{dx^2} +[\beta-\alpha -(\alpha+\beta+2)x]\,\frac{d}{dx}$$
then, by (2) above,
$$\mathcal{G}_{\alpha,\beta} P_n^{\alpha,\beta}  = -n(n+\alpha+\beta+1) P_n^{\alpha,\beta}.$$
Moreover, $\mathcal{G}_{\alpha,\beta}$ is self-adjoint with respect to the Jacobi measure $d\mu_{\alpha,\beta}$
in $(-1,1)$, that is, for sufficiently smooth functions $f$ and $g$ with compact support on $(-1,1)$, we have
$$\int_{-1}^1\mathcal{G}_{\alpha,\beta} f(x) g(x) \,d\mu_{\alpha,\beta}
=\int_{-1}^1 f(x) \mathcal{G}_{\alpha,\beta} g(x)\, d\mu_{\alpha,\beta}.$$

Now we find the appropriate factorization of $\mathcal{G}_{\alpha,\beta} $ in terms of first order differential operators.  
Let us define the operator
$$\operatorname{grad} f = (1-x^2)^{1/2} \frac{d f}{dx}$$
and observe that 
\begin{align*}
\int_{-1}^1& \operatorname{grad}f (x) g(x)\, d \mu_{\alpha,\beta} =
\int_{-1}^1 \frac{df}{dx} \Big[ g(x) (1-x)^{\alpha+1/2}(1+x)^{\beta+1/2}\Big]\,dx \\ 
&= -\int_{-1}^1  f(x) \Big[\frac{dg}{dx}(x) (1-x)^{\alpha+1/2}(1+x)^{\beta+1/2}   \\
&\qquad\qquad\qquad-\Big(\alpha+\frac{1}{2}\Big) g(x) (1-x)^{\alpha-1/2}(1+x)^{\beta+1/2}  \\
&\qquad\qquad\qquad+\Big(\beta+\frac{1}{2}\Big) g(x) (1-x)^{\alpha+1/2}(1+x)^{\beta-1/2} \Big]\,dx \\
&=-\int_{-1}^1 f(x) (1-x^2)^{1/2} \Big[\frac{dg}{dx} + \frac{-(1+x)(\alpha+1/2) +(1-x)(\beta+1/2)}{(1-x^2)} g(x) \Big]\,d\mu_{\alpha,\beta}\\
&=-\int_{-1}^1 f(x) (1-x^2)^{1/2} \Big[\frac{dg}{dx} + \frac{(\beta-\alpha) -(\alpha+\beta+1)x}{(1-x^2)} g(x) \Big]\,d\mu_{\alpha,\beta}. 
\end{align*}
We define
$$\operatorname{div}_{\alpha,\beta}g = (1-x^2)^{1/2} \Big[\frac{dg}{dx} + \frac{(\beta-\alpha) -(\alpha+\beta+1)x}{(1-x^2)} g\Big].$$
It can be verified that 
$$\operatorname{div}_{\alpha,\beta}( \operatorname{grad} f )= \mathcal{G}_{\alpha,\beta} f.$$

Within the setting of the theory developed in Section \ref{part:fractional}, we have
$$\operatorname{div}_{\alpha,\beta} f=    \mathfrak{D}_{{\rm left},(\frac{-\alpha}{1-x}+\frac{\beta}{1+x})}( (1-x^2)^{1/2} f) .$$
Hence, by choosing $x_0=0$ in \eqref{definicion}, the associated weight function is 
$$\mathcal{E}(x) = \exp\bigg[-\int_0^x\Big( \frac{-\alpha}{1-y} +\frac{\beta}{1+y}\Big)\,dy\bigg] = (1-x)^{-\alpha}(1+x)^{-\beta}.$$
As in the previous two cases (Hermite and Laguerre), we can also now write 
$$P_n^{(\alpha,\beta)}(x)=\frac{1}{2^nn!}\mathfrak{D}^n_{{\rm left},(\frac{-\alpha}{1-x}+\frac{\beta}{1+x})}((1-x^2)^n).$$

%%%%%%%%%%%%%%%%%%%%%%%%%%%%%%%%%%%%%%%%%%%%%%%%%%%%%%
\subsubsection{Lebesgue measure and transference}
%%%%%%%%%%%%%%%%%%%%%%%%%%%%%%%%%%%%%%%%%%%%%%%%%%%%%%

We define an isometry $J_{\alpha,\beta}$ from $L^2((-1,1), d\mu_{\alpha,\beta})$
into $L^2((-1,1),dx)$ as
$$J_{\alpha,\beta} f(x) = (1-x)^{\alpha/2}(1+x)^{\beta/2}  f(x).  $$
It can be checked that 
\begin{equation}\label{relacion}
\begin{aligned}
J_{\alpha,\beta}(\operatorname{grad} f) &= \frac{d}{dx}\big( (1-x^2)^{1/2} J_{\alpha,\beta} f\big)-\frac{\beta-\alpha-(\alpha+\beta+2)x}{2(1-x^2)}(1-x^2)^{1/2} J_{\alpha,\beta} f\\
J_{\alpha,\beta}(\operatorname{div}_{\alpha,\beta} f) &=
\frac{d}{dx}\big( (1-x^2)^{1/2} J_{\alpha,\beta} f\big)+\frac{\beta-\alpha-(\alpha+\beta)x}{2(1-x^2)} (1-x^2)^{1/2} J_{\alpha,\beta} f.
\end{aligned}
\end{equation}
These identities lead us to define the following operators.

\begin{definition}
For a differentiable function $g$, we let
\begin{align*}
\delta_{\alpha,\beta} g (x)&= \frac{d}{dx}\big( (1-x^2)^{1/2} g(x)\big)-\frac{\beta-\alpha-(\alpha+\beta+2)x}{2(1-x^2)} \big( (1-x^2)^{1/2} g(x)\big)\\
\delta_{\alpha,\beta}^*g (x) &=
\frac{d}{dx}\big( (1-x^2)^{1/2} g(x)\big)+\frac{\beta-\alpha-(\alpha+\beta)x}{2(1-x^2)} \big( (1-x^2)^{1/2} g(x)\big).
\end{align*}
\end{definition}

With this definition, the identities in \eqref{relacion} now read as 
\begin{equation*}
\begin{aligned}
J_{\alpha,\beta}(\operatorname{grad}f) &= \delta_{\alpha,\beta}(J_{\alpha,\beta} f) \\
J_{\alpha,\beta}(\operatorname{div}_{\alpha,\beta} f) &= \delta^*_{\alpha,\beta}(J_{\alpha,\beta} f).
\end{aligned}
\end{equation*}
We also observe that 
\begin{align*}
\delta_{\alpha,\beta}g &= \mathfrak{D}_{{\rm left}, \frac{(\alpha+\beta+2)x-(\beta-\alpha)}{2(1-x^2)}}((1-x^2)^{1/2}g)\\
-\delta^*_{\alpha,\beta}g &= \mathfrak{D}_{{\rm right}, -\frac{(\alpha+\beta)x-(\beta-\alpha)}{2(1-x^2)}}((1-x^2)^{1/2}g).
\end{align*}
Hence, the associated weights \eqref{definicion} with $x_0=0$ will be 
\begin{align*}
\mathcal{E}_{\rm left} (x) &=\exp\bigg[- \int_0^x \frac{(\alpha+\beta+2)y-(\beta-\alpha)}{2(1-y^2)}\,dy\bigg]
=(1-x)^{(\alpha+1)/2}(1+x)^{(\beta+1)/2} \\
\mathcal{E}_{\rm right} ( x) &= {\exp}\bigg[- \int_0^x -\frac{(\alpha+\beta)y-(\beta-\alpha)}{2(1-y^2)}\,dy\bigg]
= (1-x)^{-\alpha/2} (1+x)^{-\beta/2}. 
\end{align*}

We introduce the inverse Jacobi measure
$$d \mu_{-(\alpha+1),-(\beta+1)}(x) =  (1-x)^{-(1+\alpha) }(1+x)^{-(1+\beta) } dx\qquad\hbox{on}~(-1,1)$$
and define an isometry from $L^2((-1,1), d\mu_{-(\alpha+1),-(\beta+1)} )$ into $L^2((-1,1),dx )$ by
$$ J_{-(\alpha+1), -(\beta+1) }f(x) = (1-x)^{-\frac{(1+\alpha) }{2} }(1+x)^{-\frac{(1+\beta)}{2} } f(x).$$

The proofs of the next two lemmas are left to the interested reader.

\begin{lemma}\label{alphamasuno}
For a differentiable function $f$ on $(-1,1)$, we have the following identities:
$$\delta_{\alpha,\beta}( J_{-(\alpha+1), -(\beta+1) }f) =  J_{-(\alpha+1), -(\beta+1) } (\operatorname{div}_{-(\alpha+1),-(\beta+1)}f)$$
and
$$\delta^*_{\alpha,\beta}( J_{-(\alpha+1), -(\beta+1) }f) =  J_{-(\alpha+1), -(\beta+1) } (\operatorname{grad}f).$$
\end{lemma}

\begin{lemma}
For a differentiable function $g$ on $(-1,1)$, we have the following identities:
\begin{align*}
\delta_{\alpha,\beta} ( \delta_{\alpha,\beta}^*g) &=\big( (1-x^2)^{1/2} ((1-x^2)^{1/2})' \big)'+xg' \\
&\quad -\Big(\frac{\alpha}{2}+\frac{\beta}{2}\Big)g -\frac{ [(\frac{\beta}{2}- \frac{\alpha}{2})-(\frac{\alpha}{2}+\frac{\beta}{2}+1)x]^2}{1-x^2},
\end{align*}
and 
\begin{align*}
\delta_{\alpha,\beta}^*( \delta_{\alpha,\beta}g) &= \big( (1-x^2)^{1/2} ((1-x^2)^{1/2})' \big)'+x g' \\
&\quad -\Big(\frac{\alpha}{2}+\frac{\beta}{2}+1\Big)g-\frac{ [(\frac{\beta}{2}- \frac{\alpha}{2})-(\frac{\alpha}{2}+\frac{\beta}{2})x]^2}{1-x^2}.
\end{align*}
\end{lemma}

As a consequence, we have the following corollary.
\begin{corollary}\label{consecuencia}
For a differentiable function $g$ on $(-1,1)$, we have
$$\delta_{\alpha-1,\beta-1} ( \delta_{\alpha-1,\beta-1}^*g)= \delta_{\alpha,\beta}^* ( \delta_{\alpha,\beta} g)-2g.$$
\end{corollary}

%%%%%%%%%%%%%%%%%%%%%%%%%%%%%%%%%%%%%%%%%%%%%%%%%%%%%%
\subsubsection{Inverse Jacobi measure and new polynomials of Jacobi type.
Proofs of Theorem \ref{inverJac} and Theorem \ref{inverAll}$(iii)$}
%%%%%%%%%%%%%%%%%%%%%%%%%%%%%%%%%%%%%%%%%%%%%%%%%%%%%%

By using Lemma \ref{alphamasuno} and Corollary \ref{consecuencia}, we get
\begin{equation}\label{relacion2}
\begin{aligned}
\operatorname{div}& _{-(\alpha+1),-(\beta+1)}(\operatorname{grad}f)
= J_{-(\alpha+1),-(\beta+1)}^{-1}\delta_{\alpha,\beta}\delta^*_{\alpha,\beta}J_{-(\alpha+1),-(\beta+1)} f \\
 &= J_{-(\alpha+1),-(\beta+1)}^{-1}\delta^*_{\alpha+1,\beta+1}\delta_{\alpha+1,\beta+1} J_{-(\alpha+1),-(\beta+1)} f -2f\\
 &= J_{-(\alpha+1),-(\beta+1)}^{-1}J_{\alpha+1,\beta+1}\operatorname{div}_{\alpha+1,\beta+1}\operatorname{grad} J_{\alpha+1,\beta+1}^{-1}J_{-(\alpha+1),-(\beta+1)} f -2f.
\end{aligned}
\end{equation}
In view of the identities in \eqref{relacion2}, we define the operator $\widetilde{\mathcal{G}}_{\alpha,\beta} $ as
$$\widetilde{\mathcal{G}}_{\alpha,\beta} = (\operatorname{div}_{-\alpha,-\beta}) (\operatorname{grad}) =
(1-x^2) \frac{d^2}{dx^2}+ \big( ( \alpha-\beta) + (\alpha+\beta-2)x\big) \frac{d}{dx}.$$
Then the functions
\begin{equation}\label{relationJacobi}
P_n^{(\alpha,\beta), *} (x)= (1-x)^\alpha (1+x)^\beta  P_n^{(\alpha,\beta)} (x),\qquad\hbox{for}~x\in(-1,1),~n\geq0,
\end{equation}
form an orthogonal family of eigenfunctions of the  operator $\widetilde{\mathcal{G}}_{\alpha,\beta}$
with respect to the measure $d\mu_{-\alpha, -\beta}$ and corresponding eigenvalues $-n(n+\alpha+\beta+1) -2$. 

Next we present the new family of polynomials that are eigenfunctions of the operator 
$\widetilde{\mathcal{G}}_{\alpha,\beta}$. Consider, for $x\in(-1,1)$ and $|t|<1$, the function
$$\omega(x,t)= \frac{(1-t+(1-2xt+t^2)^{1/2})^{\alpha} (1+t+(1-2xt+t^2)^{1/2})^{\beta}}{2^{\alpha+\beta}(1-2xt+t^2)^{1/2}}
=\sum_{n=0}^\infty c_n(x)t^n.$$

As this function is analytic in a ball of  center $x$ and radius  small enough, we can use Cauchy's Theorem to get
$$c_n(x) = \frac1{2\pi i}\int_C \omega(x,t)\,   t^{-n-1}\,dt$$
where the contour $C$ is a circle centered at $x$ with small radius.
We make the change of variables $1-rt=  (1-2xt+t^2)^{1/2}$. Then  
$$t= \frac{2(x-r)}{1-r^2},\qquad 1-rt=\frac{1+u^2-2xr}{1-r^2},$$
$$1+t +  (1-2xt+t^2)^{1/2} = \frac{2(1+x)}{1+r},$$
$$1-t +  (1-2xt+t^2)^{1/2} =\frac{2(1-x)}{1-r},$$
and
$$dt= 2\frac{-1-r^2 +2xr}{(1-r^2)^2}.$$
Hence,
\begin{align*}
c_n&(x)=\frac1{2\pi i}\int_C \omega(x,t)t^{-n-1}\,dt \\
&= \frac{(-1)^{n+1}}{2^{\alpha+\beta}\pi i}
\int_{C'} \frac{1-r^2}{1+r^2-2xr}\Big[\frac{2(1+x)}{1+r}\Big]^{\alpha}\Big[\frac{2(1-x)}{1-r}\Big]^{\beta}
\Big[\frac{2(r-x)}{1-r^2}\Big]^{-n-1}\frac{(-1-r^2 +2xr)}{(1-r^2)^2}\,dr \\
&= (-1)^n(1-x)^{\alpha}(1+x)^{\beta}\frac{2^{-n}}{2\pi i} 
\int_{C'}  \frac{(1-r)^{-\alpha+n}(1+r)^{-\beta+n}}{(r-x)^{n+1}}\, dr  \\
&=\frac{(-1)^{n}}{2^nn!} (1-x)^{\alpha} (1+x)^{\beta}\frac{d^n}{dx^n} \big( (1-x)^{-\alpha+n}(1+x)^{-\beta+n}\big).
\end{align*}
We then define the family of polynomials 
\begin{equation}\label{inverjacobi}
\widetilde{P}_n^{\alpha,\beta}(x) =  \frac{(-1)^{n}}{2^nn!} (1-x)^{\alpha} (1+x)^{\beta}\frac{d^n}{dx^n}
\big( (1-x)^{-\alpha+n}(1+x)^{-\beta+n}\big)
\end{equation}
and show that
\begin{equation}\label{autovaloresinversos3}
\widetilde{\mathcal{G}}_{\alpha,\beta}\widetilde{P}_n^{\alpha,\beta}=-n(n-\alpha-\beta+1)\widetilde{P}_n^{\alpha,\beta}.
\end{equation}
Indeed, it can be checked that $y(x)$ satisfies the ODE
\begin{equation}\label{difeqGe}
(1-x^2)y''+[-\beta+\alpha +(\alpha+\beta-2)x]\,y'+n(n-\alpha-\beta+1)y=0
\end{equation}
if and only if $\displaystyle Y(x)= (1-x)^{-\alpha}(1+x)^{-\beta} y(x)$ solves
\begin{equation}\label{difeqGe2}
(1-x^2)Y''+[-\alpha+\beta-(\alpha+\beta+2)x]\,Y'+(n+1)(n-\alpha-\beta)Y=0,
\end{equation}
see \cite[p.~61]{Sz}. Thus, we have to verify \eqref{difeqGe2} in the case when
$$Y(x)= (1-x)^{-\alpha}(1+x)^{-\beta} \widetilde{P}_n^{\alpha,\beta}(x)= \frac{(-1)^{n}}{2^{n+1}\pi i} 
\int_C \frac{(1-r)^{-\alpha+n}(1+r)^{-\beta+n}}{(r-x)^{n+1}}\,dr.$$
We observe that 
$$Y'(x)=(n+1)\frac{(-1)^{n}}{2^{n+1}\pi i}\int_C  \frac{(1-r)^{-\alpha+n}(1+r)^{-\beta+n}}{(r-x)^{n+2}}\,dr$$
and
$$Y''(x) = (n+1)(n+2) \frac{(-1)^{n}}{2^{n+1}\pi i}\int_C  \frac{(1-r)^{-\alpha+n}(1+r)^{-\beta+n}}{(r-x)^{n+3}}\,dr.$$
Hence 
\begin{align*}
(1-x^2)&Y''+[-\alpha+\beta-(\alpha+\beta+2)x]\,Y'+(n+1)(n-\alpha-\beta)Y\\
&= \frac{(-1)^{n}}{2^{n+1}\pi i} 
\int_C  \frac{(1-r)^{-\alpha+n}(1+r)^{-\beta+n}}{(r-x)^{n+3}}\times \\
&\qquad\times  \big\{ (n+1)(n+2) (1-x^2)  +[-\alpha+\beta-(\alpha+\beta+2)x](n+1)(r-x) \\
&\qquad\qquad+(n+1)(n-\alpha-\beta)(r-x)^2\big\}\,dr\\
&= (n+1)\frac{(-1)^{n}}{2^{n+1}\pi i} \int_C  \frac{(1-r)^{\alpha+n}(1+r)^{\beta+n}}{(r-x)^{n+3}} \times \\
&\qquad\times \big\{n+2+(\alpha-\beta)r +(\alpha-\beta)x+(-2-2n+\alpha+\beta)rx+ (n-\alpha-\beta)r^2 \big\}\,dr\\ 
&= -(n+1)\frac{(-1)^{n}}{2^{n+1}\pi i} \int_C  \frac{d}{dr} \frac{(1-r)^{-\alpha+n+1}(1+r)^{-\beta+n+1}}{(r-x)^{n+2}}\, dr= 0.
\end{align*}
We conclude by showing the recurrence relation for the polynomials 
$\{\widetilde{P}_n^{\alpha,\beta}\}_{n\geq0}$. Before that, we prove the following.

\begin{lemma}\label{expresionalternativa}
The following equality holds
$$
\widetilde{P}_n^{\alpha,\beta}(x)=\sum_{k=0}^n {n-\alpha \choose k}{n-\beta \choose n-k} \Big(\frac{x-1}{2} \Big)^{n-k}\Big(\frac{x+1}{2} \Big)^{k}.
$$
In particular, $\widetilde{P}_n^{\alpha,\beta}(1)={n-\alpha \choose n}$ and $\widetilde{P}_n^{\alpha,\beta}(-1)=(-1)^n{n-\beta \choose n}$.
\end{lemma}

\begin{proof}
By \eqref{inverjacobi},
$$
\widetilde{P}_n^{\alpha,\beta}(x) =  \frac{(-1)^{n}}{2^nn!} (1-x)^{\alpha} (1+x)^{\beta}\frac{d^n}{dx^n}
\big( (1-x)^{-\alpha+n}(1+x)^{-\beta+n}\big).
$$
Using Leibniz formula we can write
\begin{align*}
\frac{d^n}{dx^n}&\Big( (1-x)^{-\alpha+n}(1+x)^{-\beta+n}\Big) = \sum_{k=0}^n {n \choose k}\frac{d^k}{dx^k}
\Big( (1-x)^{-\alpha+n}\Big)\, \frac{d^{n-k}}{dx^{n-k}}
\Big( (1+x)^{-\beta+n}\Big) \\
&= \sum_{k=0}^n {n \choose k} \Big\{ I_{\{0<k< n\}}(n-\alpha)(n-1-\alpha)\dots(n-(k-1)-\alpha) (-1)^k(1-x)^{-\alpha+n-k} \\
&\qquad\qquad\qquad\times (n-\beta)(n-1-\beta)\dots(n-(n-k-1)-\beta) (1+x)^{-\beta+k} \Big\}\\
&\quad+(n-\alpha)(n-1-\alpha)\dots(n-(n-1)-\alpha) (-1)^n(1-x)^{-\alpha}(1+x)^{-\beta+n}\\
&\quad+(1-x)^{-\alpha+n}(n-\beta)(n-1-\beta)\dots(n-(n-1)-\beta) (1+x)^{-\beta},
\end{align*}
from which it follows that
\begin{align*}
\widetilde{P}_n^{\alpha,\beta}(x) 
 &=  \frac{1}{n!} 
 \sum_{k=0}^n {n \choose k} \Big(\frac{x-1}{2}\Big)^{n-k} \Big(\frac{1+x}{2}\Big)^k\\
 &\qquad\qquad\times \Big\{ I_{\{0<k< n\}}(n-\alpha)(n-1-\alpha)\dots(n-(k-1)-\alpha)\\
&\qquad\qquad\qquad\times (n-\beta)(n-1-\beta)\dots(n-(n-k-1)-\beta)  \Big\}\\
&\qquad+ \frac1{n!}\Big(\frac{1+x}{2}\Big)^n(n-\alpha)(n-1-\alpha)\dots(n-(n-1)-\alpha)\\
&\qquad+\frac1{n!}\Big(\frac{x-1}{2}\Big)^n (n-\beta)(n-1-\beta)\dots(n-(n-1)-\beta)  \\
&=\sum_{k=0}^n {n-\alpha \choose k}{n-\beta \choose n-k} \Big(\frac{x-1}{2} \Big)^{n-k}\Big(\frac{x+1}{2} \Big)^{k}.
\end{align*}
\end{proof}

Next, the change of variables $x=1-2x'$ in \eqref{difeqGe} yields
$$
x'(1-x')y''+[(1-\alpha)+(\alpha+\beta-2)x']y'+n(n-\alpha-\beta+1)y=0,
$$
which is the hypergeometric equation of Gauss. The polynomial solution of this equation can be represented in terms of the hypergeometric function $F$. Taking into account that $\widetilde{P}_n^{\alpha,\beta}(1)={n-\alpha \choose n}$, we have that
$$
\widetilde{P}_n^{\alpha,\beta}(x)={n-\alpha \choose n} F\left(-n, n-\alpha-\beta+1; 1-\alpha; \frac{1-x}{2}\right).
$$

We can now find the constants $A_0, A_1, B$ such that
\begin{equation}\label{three term recurrence}
\widetilde{P}_{n+1}^{\alpha,\beta}(x)=(A_1 x + A_0)\widetilde{P}_n^{\alpha,\beta}(x) + B \widetilde{P}_{n-1}^{\alpha,\beta}(x).
\end{equation}
We can actually assume that these constants exist since this kind of recurrence relation
always holds for the hypergeometric function (see, for instance, \cite{YaDeNi}). Since
$$
F\left(-n, n-\alpha-\beta+1; 1-\alpha; \frac{1-x}{2}\right)=\sum_{k=0}^n(-1)^k {n\choose k}\frac{(n-\alpha-\beta+1)_k}{(1-\alpha)_k}\left(\frac{1-x}{2}\right)^k,
$$
where $(q)_k$ is the Pochhammer symbol defined by $(q)_k={q+k-1 \choose k} k!$, we see that
$$
\lim_{x\to \infty} \frac{F\left(-n, n-\alpha-\beta+1; 1-\alpha; \frac{1-x}{2}\right)}{x^n}=\frac{(n-\alpha-\beta+1)_n}{2^n (1-\alpha)_n},
$$
which yields to
$$
\lim_{x\to \infty} \frac{\widetilde{P}_n^{\alpha,\beta}(x)}{x^n}=\frac{1}{2^n} {2n-\alpha-\beta \choose n}.
$$
Then, dividing by $x^{n+1}$ and taking limit in \eqref{three term recurrence} we obtain
$$
\frac{1}{2^{n+1}}{2n-\alpha-\beta+2 \choose n+1}=\lim_{x\to \infty} \frac{\widetilde{P}_{n+1}^{\alpha,\beta}(x)}{x^{n+1}}= A_1 \lim_{x\to \infty} \frac{\widetilde{P}_n^{\alpha,\beta}(x)}{x^n}=A_1 \frac{1}{2^n} {2n-\alpha-\beta \choose n},
$$
from where we deduce
$$
A_1=\frac{(2n-\alpha-\beta+2)(2n-\alpha-\beta+1)}{2(n+1)(n-\alpha-\beta+1)}.
$$
Now, evaluating \eqref{three term recurrence} at $x=1$ and $x=-1$, having in mind Lemma \ref{expresionalternativa}, we obtain
$$
{n+1-\alpha \choose n+1}=(A_1+A_0){n-\alpha \choose n} + B {n-1-\alpha \choose n-1}
$$
and 
$$
(-1)^{n+1}{n+1-\beta \choose n+1}=(-A_1+A_0)(-1)^n{n-\beta \choose n} + B (-1)^{n-1} {n-1-\beta \choose n-1},
$$
from which it follows that
$$
A_0=\frac{(2n-\alpha-\beta+1)(\alpha^2-\beta^2)}{2(n+1)(n-\alpha-\beta+1)(2n-\alpha-\beta)}
$$
and
$$
B=-\frac{(n-\alpha)(n-\beta)(2n-\alpha-\beta+2)}{(n+1)(n-\alpha-\beta+1)(2n-\alpha-\beta)}.
$$
This proves the three-term recurrence formula stated in Theorem~\ref{inverJac}.

Parallel to the case of Hermite and Laguerre, observe that the operator $\widetilde{\mathcal{G}}_{\alpha,\beta}$ is self-adjoint with respect to the inverse Jacobi measure   $(x-1)^{-\alpha} (x+1)^{-\beta}dx$.
Furthermore, the family of polynomials $\widetilde{P}_n^{\alpha,\beta}$ that we have found here,
even though they satisfy \eqref{autovaloresinversos3}, they do not belong to the space $L^2((-1,1),(x-1)^{-\alpha} (x+1)^{-\beta}dx)$.  Also the family  $\{ P_n^{(\alpha,\beta), *} \}$ defined in (\ref{relationJacobi}), is an orthonormal basis of $L^2( (-1,1), (x-1)^{-\alpha}(x+1)^{-\beta}).$

%%%%%%%%%%%%%%%%%%%%%%%%%%%%%%%%%%%%%%%%%%%%%%%%%%%%%%
\subsection{Favard-type theorem}
%%%%%%%%%%%%%%%%%%%%%%%%%%%%%%%%%%%%%%%%%%%%%%%%%%%%%%

In order to prove  Theorem \ref{favard} we need the following result that can be found, for instance, in 
\cite[p.~21]{Chihara}. 
  
\begin{theorem}[Classical Favard's Theorem]
Let $\{c_n\}_{n=1}^\infty$ and $\{\lambda_n\}_{n=1}^\infty$ be  sequences of complex numbers with $\lambda_n\neq 0$ and let $\{P_n(x)\}_{n=0}^\infty$ be a sequence of monic polynomials defined by the recurrence formula 
 $$\begin{cases}
 P_{-1}(x)= 0\\
 P_0(x)=1\\
 P_n(x)=(x-c_n) P_{n-1}(x)-\lambda_n P_{n-2}(x),&\hbox{for}~n\geq1.
\end{cases}$$
Then there is a unique moment functional $\mathcal{L}$ such that 
$$\mathcal{L}[1] = \lambda_1, \quad \mathcal{L}[P_m(x) P_n(x)] = 0,  \quad \mathcal{L}[P_n^2(x) ] \neq 0,  \quad
\hbox{for}~m,n\geq0,~m\neq n.$$
In other words, $\mathcal{L}$ is quasi-definite and $\{P_n(x)\}$ is the corresponding sequence
of monic polynomials associated to $\mathcal{L}$. 
\end{theorem}  

\begin{proof}[Proof of Theorem \ref{favard}.]
It is enough to verify that our (monic) polynomials satisfy the recurrence formula
in the statement of Favard's theorem.
Let $\widetilde{Q}_n$ be the monic polynomials of the Hermite family, so that
$\widetilde{H}_n =(-1)^n 2^n \widetilde{Q}_n$. Hence, by the recurrence formula of $\widetilde{H}_n$,
$$\widetilde{Q}_{n+1}(x)   = x  \widetilde{Q}_n(x)  -\frac12 n \widetilde{Q}_{n-1}(x).$$
Similarly, if $\widetilde{Q}_n$ are the monic polynomials of the Laguerre family, we have
$\widetilde{L}_n =\frac1{n!}\widetilde{Q}_n$, and the recurrence formula for $\widetilde{L}_n$ implies that   
$$\widetilde{Q}_{n+1}(x)  =-(\alpha-1-x-2n) \widetilde{Q}_n(x)  -(n-\alpha)n \widetilde{Q}_{n-1}(x).$$
Finally, let us denote by $B_n$ the coefficient of $x^n$ in $\widetilde{P}_{n}^{\alpha,\beta}(x).$ It is easy to see that  these coefficients satisfy the recurrence relation
$$B_n =B_{n-1} \frac{(2n-\alpha-\beta)(2n-\alpha-\beta-1)}{2n(n-\alpha-\beta-2)}, \quad B_1=1.$$
Hence, the (rather cumbersome) recurrence formula for the family $\{\widetilde{Q}_n\}_n$
of monic polynomials associated to the Jacobi family can be verified.
\end{proof}

%%%%%%%%%%%%%%%%%%%%%%%%%%%%%%%%%%%%%%%%%%%%%%%%%%%%%%
\subsection{Harmonic analysis applications.}
%%%%%%%%%%%%%%%%%%%%%%%%%%%%%%%%%%%%%%%%%%%%%%%%%%%%%%

In this final section we prove that the Hermite Riesz transforms $R_if(x) = \partial_{x_i}\widetilde{\mathcal{O}}^{-1/2} f(x)$,
$i=1,\ldots,d$, are bounded on $L^2(\gamma_{-d})$.
The rest of the proofs of Theorems~\ref{aplicaGauss} and \ref{aplicaLaguerre} follow analogous ideas 
and details are left to the interested reader.

\begin{proof}[Proof of Theorem \ref{aplicaGauss}~$\mathrm{b)}$]
It is easy to check that \eqref{relationtilde} can be extended to higher dimensions.
In fact, for $\widetilde{\mathcal{O}}= \Delta+2x\cdot \nabla$ and 
 $\mathcal{O}= -\Delta+2x\cdot \nabla$, $x \in \mathbb{R}^d $, we have
 $$\widetilde{\mathcal{O}} f(x) = -e^{-|x|^2} \mathcal{O}(e^{|\cdot|^2} f(\cdot))(x)-2d\,f(x).$$
This gives
$$e^{-t \widetilde{\mathcal{O}} } f(x) = -e^{-|x|^2}e^{-t (\mathcal{O}+2d\,I_d)}(e^{|\cdot|^2} f(\cdot))(x).$$
Hence, for $\sigma>0$ and sufficiently smooth functions $f$, we have 
$$\widetilde{\mathcal{O}}^{-\sigma} f(x) = -e^{-|x|^2}(\mathcal{O}+2d\,I_d)^{-\sigma}(e^{|\cdot|^2} f(\cdot))(x).$$
Also, 
$$(\partial_{x_i}+2x_i) \widetilde{\mathcal{O}}^{-1/2} f(x) = -e^{-|x|^2}\partial_{x_i}(\mathcal{O}+2d\,I_d)^{-1/2}(e^{|\cdot|^2} f(\cdot))(x),$$
$$\partial_{x_i}\widetilde{\mathcal{O}}^{-1/2} f(x) = -e^{-|x|^2}(\partial_{x_i}-2x_i)(\mathcal{O}+2d\,I_d)^{-1/2}(e^{|\cdot|^2} f(\cdot))(x)$$
and
$$\partial_te^{-t \widetilde{\mathcal{O}} } f(x) = 
-e^{-|x|^2}\partial_t e^{-t (\mathcal{O}+2d\,I_d)}(e^{|\cdot|^2} f(\cdot))(x).$$
Let $\Lambda$ be the multiplier $\Lambda =\mathcal{O}^{1/2} (\mathcal{O}+2d\,I_d)^{-1/2}$,
which is defined in the usual spectral way using the classical Hermite polynomials expansion. We have 
\begin{align*}
R_if(x)&=\partial_{x_i}\widetilde{\mathcal{O}}^{-1/2} f(x) =-e^{-|x|^2}(\partial_{x_i}-2x_i)(\mathcal{O}+2d\,I_d)^{-1/2}(e^{|\cdot|^2} f(\cdot))(x)\\
&= -e^{-|x|^2}(\partial_{x_i}-2x_i)\mathcal{O}^{-1/2} \Lambda (e^{|\cdot|^2} f(\cdot))(x).
\end{align*} 
If $f\in L^2(\gamma_{-d})$ then $e^{|\cdot|^2 } f(\cdot) \in L^2(\gamma_d)=L^2(\pi^{-d/2}e^{-|x|^2}dx)$. On the other hand the multiplier $\Lambda$ is bounded in  $L^2(\gamma )$. It is well known that the operator $(\partial_{x_i}-2x_i)(\mathcal{O})^{-1/2} $ is bounded on $L^2(\gamma_d)$.
Finally, if $g\in L^2(\gamma_d)$ then $e^{-|\cdot|^2 }g(\cdot) \in L^2(\gamma_{-d})$.
Thus, $R_i$ is bounded in $L^2(\gamma_{-d}).$ The proofs for $R_i^*$ is analogous. 
\end{proof}

\medskip

\noindent\textbf{Acknowledgments.} We are grateful to Manuel Alfaro (Universidad de Zaragoza)
and Xuan Hien Nguyen (Iowa State University) for helpful remarks.

%%%%%%%%%%%%%%%%%%%%%%%%%%%%%%%%%%%%%%%%%%%%%%%%%%%%%%

%%%%%%%%%%%%%%%%%%%%%%%%%%%%%%%%%%%%%%%%%%%%%%%%%%%%%%

%%%%%%%%%%%%%%%%%%%%%%%%%%%%%%%%%%%%%%%%%%%%%%%%%%%%%%
\end{document}